\documentclass[12pt]{amsart}

\usepackage{amsmath}
\usepackage{amsthm}
\usepackage{amssymb}
\usepackage{mathabx}
\usepackage{mathrsfs} % for \mathscr{...} font
\usepackage[all]{xy} % commutative diagrams
\usepackage{graphicx}
\graphicspath{{figures/}} % figures kept in subdirectory or folder

\usepackage{cancel} % $\cancel{x}$ to get a slash through x in math mode
\usepackage{enumerate} % provides useful extra options for the enumerate environment
\usepackage{hyperref} % for links within the document
\usepackage[usenames,dvipsnames]{xcolor}

\usepackage{tikz}
\usetikzlibrary{arrows}
\usetikzlibrary{calc}
\usetikzlibrary{shapes}
\usetikzlibrary{decorations.pathmorphing}

\hypersetup{colorlinks=true, linkcolor=NavyBlue, urlcolor=WildStrawberry}

\colorlet{Darkyellow}{yellow!80!magenta}

\makeatletter
\newtheorem*{rep@theorem}{\rep@title}
\newcommand{\newreptheorem}[2]{%
\newenvironment{rep#1}[1]{%
 \def\rep@title{#2 \ref{##1}}%
 \begin{rep@theorem}}%
 {\end{rep@theorem}}}
\makeatother

\makeatletter
\def\blfootnote{\xdef\@thefnmark{}\@footnotetext}
\makeatother

\newreptheorem{theorem}{Theorem}
\newreptheorem{lemma}{Lemma}

\theoremstyle{plain} % text is in italics
\newtheorem{theorem}{Theorem}[section]
\newtheorem{lemma}[theorem]{Lemma}
\newtheorem{cor}[theorem]{Corollary}
\newtheorem*{theorem*}{Theorem}
\newtheorem*{lemma*}{Lemma}
\newtheorem*{cor*}{Corollary}

\theoremstyle{definition} % text is not in italics
\newtheorem{definition}[theorem]{Definition}
\newtheorem{notation}[theorem]{Notation}
\newtheorem{example}[theorem]{Example}
 
\newtheorem{remark}[theorem]{Remark} 
 
\newtheorem*{definition*}{Definition}
\newtheorem*{notation*}{Notation}
\newtheorem*{example*}{Example}
\newtheorem*{remark*}{Remark}
\newtheorem*{question*}{Question}

\DeclareMathOperator{\Out}{Out}
\DeclareMathOperator{\Aut}{Aut}
\DeclareMathOperator{\Inn}{Inn}
\DeclareMathOperator{\Spe}{Spe}
\DeclareMathOperator{\cat}{CAT}
\DeclareMathOperator{\Isom}{Isom}
\DeclareMathOperator{\GL}{GL}
\DeclareMathOperator{\HTc}{\mathcal{HT}}
\DeclareMathOperator{\HT}{\mathrm{HT}}
\DeclareMathOperator{\Kc}{\mathcal{K}}
\DeclareMathOperator{\K}{\mathrm{K}}
\DeclareMathOperator{\Lk}{Lk} % the link
\DeclareMathOperator{\St}{St} % the star
\DeclareMathOperator{\LP}{(}
\DeclareMathOperator{\RP}{)}

\newcommand\numberthis{\addtocounter{equation}{1}\tag{\theequation}}

\title[Geometry of OUT($W_n$)]{The Geometry of Outer Automorphism Groups of Universal Right-Angled Coxeter Groups}
\author{Charles Cunningham}

\let\oldtocsection=\tocsection
\let\oldtocsubsection=\tocsubsection
\renewcommand{\tocsection}[2]{\hspace{0em}\oldtocsection{#1}{#2}}
\renewcommand{\tocsubsection}[2]{\hspace{1.1em}\oldtocsubsection{#1}{#2}}

\begin{document}

% Note: the abstract precedes the \maketitle in AMS document classes.
\begin{abstract}
We investigate the combinatorial and geometric properties of automorphism groups of universal right-angled Coxeter groups, which are the automorphism groups of free products of copies of $\mathbb{Z}_2$. It is currently an open question as to whether or not these automorphism groups have non-positive curvature. Analogous to Outer Space as a model for $\Out(F_n)$, we prove that the natural combinatorial and topological model for their outer automorphism groups can \emph{not} be given an equivariant $\cat(0)$ metric. This is particularly interesting as there are very few non-trivial examples of proving that a model space of independent interest is \emph{not} $\cat(0)$. 
\end{abstract}

\maketitle

\tableofcontents

%% In the amsart document class, paragraphs are indented.  Instead, you may want paragraphs to be separated by an extra line and NOT indented.  If this is the case, uncomment the lines below.
%% (These must come AFTER the \tableofcontents.)
%\setlength{\parindent}{0cm} % no indentation
%\setlength{\parskip}{\baselineskip} % extra line between paragraphs
%\setlength{\abovedisplayskip}{0.7\baselineskip} % fix spacing around displayed math
%\setlength{\belowdisplayskip}{0.7\baselineskip}

\section{Introduction}\label{sec:intro}\blfootnote{Portions of this paper originally appeared in the author's Ph.D. thesis \cite{mythesis}.}
Right-angled Coxeter groups are a robust and interesting class of examples of non-positively curved groups that generalize groups generated by reflections in Euclidean and hyperbolic spaces. They are all known to be $\cat(0)$, a form of non-positive curvature that is shared by free groups, free abelian groups, and fundamental groups of closed Euclidean and hyperbolic manifolds. But while the automorphism groups $\Aut(F_n)$ and $\Out(F_n)$ for $n \geq 3$ are known to \emph{not} be $\cat(0)$ themselves \cite{Gersten,BV}, this is still open for the analogous automorphism groups of the universal right-angled Coxeter groups $\Aut(W_n)$ and $\Out(W_n)$, where $W_n$ is a free product of copies of $\mathbb{Z}_2$ instead of $\mathbb{Z}$. 

There is a natural contractible combinatorial and topological model of $\Out(F_n)$ called \emph{Outer Space}, $X_n$, in the sense that the simplicial automorphisms of $X_n$ are precisely $\Out(F_n)$ \cite{BV2}, although Bridson \cite{Bridson} showed that $X_n$ could not be given an $\Out(F_n)$-equivaraint $\cat(0)$ metric (before it was known that $\Out(F_n)$ isn't a $\cat(0)$ group at all). 

We now find ourselves in an analogous position in the right-angled Coxeter case. There is a natural contractible combinatorial and topological model of $\Out(W_n)$ called \emph{McCullough-Miller Space}, $K_n$, \cite{MM} in the sense that the simplicial automorphisms of $K_n$ are precisely $\Out(W_n)$ \cite{Piggott}, and it was unknown whether or not this space could be given an equivariant $\cat(0)$-metric. 

Following Bridson, we prove the following theorem:

\begin{reptheorem}{thm:noout0n}
There does not exist an $\Out^0(W_n)$-equivariant (or $\Out(W_n)$-equivariant) piecewise Euclidean (or piecewise hyperbolic) $\cat(0)$ $\LP\cat(-1)\RP$ metric on $\K_n$ for $n \geq 4$. 
\end{reptheorem}

It is thus still an open question as to whether or not $\Out(W_n)$ and $\Aut(W_n)$ are $\cat(0)$ groups, but another model space will be needed if they are. 

\section{Background}\label{sec:background}
A \emph{$\cat(0)$ metric space} is a geodesic metric space such that geodesic triangles are no fatter than corresponding Euclidean triangles with the same side lengths. If a finitely generated group $G$ acts on a $\cat(0)$ space $X$ properly discontinuously, co-compactly, and by isometries, then $G$ is called a \emph{CAT(0) group}. The standard reference on $\cat(0)$ groups and spaces is \cite{BH}. 

$\cat(0)$ groups are a generalized notion of non-positive curvature for groups. Unlike Gromov's $\delta$-hyperbolic groups, the property of being a $\cat(0)$ group is not a quasi-isometric invariant \cite{KL,BH}. Furthermore, even if a group has a natural geometric model, the failure of that model to be $\cat(0)$ doesn't preclude the possibility of the group acting geometrically on a different metric space which is $\cat(0)$. Thus, it can be a more subtle question to determine when a group is $\cat(0)$ or not.

In the 1930s, H.S.M. Coxeter introduced abstract Coxeter groups as a generalization of groups generated by geometric reflections. Their subsequent study has connected many areas of algebra, geometry, and combinatorics. 

\begin{definition*}
Given a finite simple graph $\Gamma$, the \emph{right-angled Coxeter group defined by $\Gamma$} is the group $W = W_{\Gamma}$ generated by the vertices of $\Gamma$.  The relations of $W_{\Gamma}$ declare that the generators all have order 2, and adjacent vertices in $\Gamma$ commute with each other. 
\end{definition*}

Right-angled Coxeter groups (commonly abbreviated RACGs) have a rich combinatorial and geometric history. They each act properly discontinuously and cocompactly by isometries on a metric space, called a Davis complex \cite{Davis}. Gromov \cite{Gromov} showed this space to be $\cat(0)$ for RACGs, and Moussong showed \cite{Moussong} that all Coxeter groups are in fact $\cat(0)$ groups.

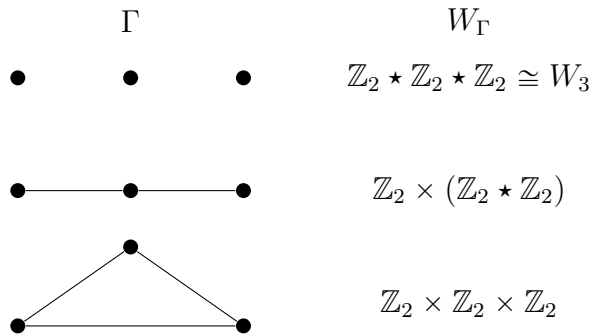
\begin{figure}[h]
\begin{center}
\begin{tikzpicture}[scale=1.5]

\node at (-2,0.5) {$\Gamma$};
\node at (1,0.5) {$W_\Gamma$};

\node (a1) [circle,fill,inner sep=2pt] at (-3,0) {};
\node (a2) [circle,fill,inner sep=2pt] at (-2,0) {};
\node (a3) [circle,fill,inner sep=2pt] at (-1,0) {};

\node at (1,0) {$\mathbb{Z}_2 \star \mathbb{Z}_2 \star \mathbb{Z}_2 \cong W_3$};

\node (b1) [circle,fill,inner sep=2pt] at (-3,-1) {};
\node (b2) [circle,fill,inner sep=2pt] at (-2,-1) {};
\node (b3) [circle,fill,inner sep=2pt] at (-1,-1) {};

\draw (b1) to (b2);
\draw (b2) to (b3); 

\node at (1,-1) {$\mathbb{Z}_2 \times \left(\mathbb{Z}_2 \star \mathbb{Z}_2\right)$};

\node (c1) [circle,fill,inner sep=2pt] at (-3,-2.2) {};
\node (c2) [circle,fill,inner sep=2pt] at (-2,-1.5) {};
\node (c3) [circle,fill,inner sep=2pt] at (-1,-2.2) {};

\draw (c1) to (c2);
\draw (c2) to (c3); 
\draw (c1) to (c3);

\node at (1,-2) {$\mathbb{Z}_2 \times \mathbb{Z}_2 \times \mathbb{Z}_2$};

\end{tikzpicture}
\caption[Defining Graphs of RACGs]{Some examples of defining graphs $\Gamma$ and their RACGs $W_\Gamma$.}
\label{fig:RACG}
\end{center}
\end{figure}

The combinatorial nature of RACGs makes them useful in studying their $\cat(0)$ geometry as they admit a biautomatic structure as well as a geodesic normal form. Thus, they have effective solutions to the word and conjugacy problems. They are also \emph{rigid}, which means a given RACG cannot arise from two different defining graphs \cite{Green,Droms,Laurence,Radcliffe}. Thus, all of the combinatorial information of the group is contained in the graph $\Gamma$. 

\begin{example*}
One particularly interesting class of examples is the universal right-angled Coxeter groups, $W_n$, whose defining graph is the empty graph on $n$ vertices. For instance, the group

$W_4 = \left\langle a_1, a_2, a_3, a_4 \mid a_1^2 = a_2^2 = a_3^2 = a_4^2 = 1 \right\rangle$ is a right-angled Coxeter group and so is $\cat(0)$. Note that $\displaystyle W_n \cong \Asterisk_{i=1}^n \mathbb{Z}_2$. 
\end{example*}

The automorphisms of right-angled Coxeter groups are generated by automorphisms that come in three varieties \cite{CRSV,Green,Lauthesis}:
\begin{enumerate}
\item Graph symmetries, which are automorphisms of $W_{\Gamma}$ induced by graph automorphisms of $\Gamma$. For instance, if two vertices of $\Gamma$ are adjacent to the same set of vertices, then $W_{\Gamma}$ has an automorphism which exchanges those two generators and leaves all other generators fixed.

\item Partial Conjugations, which conjugate a certain set of generators, $D$, by a particular generator $a_i$ while leaving all other generators fixed. The combinatorics of $\Gamma$ constrain which subsets $D$ of the generators result in automorphisms of $W_{\Gamma}$ for each $a_i$. 

\item Transvections, which send $a_i$ to $a_i a_j$ for a particular pair of generators and leave all other generators fixed. 
\end{enumerate}

\begin{definition}\label{def:pc}
Following \cite{Piggott}, we denote by $x_{i,D}$ the partial conjugation of $W_\Gamma$ defined by:
\[
a_j \mapsto 
\begin{cases}
a_i a_j a_i &\mbox{if } j \in D \\
a_j &\mbox{if } j \in [n] \setminus D
\end{cases}
\]
We call $x_{i,D}$ the \emph{partial conjugation with acting letter $a_i$ and domain $D$}.

If $\St(a_i)$ is the star of the vertex $a_i$ in $\Gamma$, then $x_{i,D}$ is an automorphism of $W_n$ if and only if $D$ is a union of connected components of $\Gamma \setminus \St(a_i)$. 

When $D$ is a single connected component of $\Gamma \setminus \St(a_i)$, we follow \cite{CEPR} and call $x_{i,D}$ an \emph{elementary partial conjugation}. 
\end{definition}

Any automorphism of a group must send involutions to involutions, and the only involutions of $W_\Gamma$ are conjugates of commuting products of its generators \cite{Bourbaki}. Furthermore, no commuting products of generators are conjugate to one another in $W_\Gamma$ \cite{Davis}, and so any automorphism of $W_\Gamma$ must permute the conjugacy classes of commuting products of the generators. Thus, $\Aut(W_\Gamma)$ acts on the set of conjugacy classes of commuting products of the generators, whose kernel is denoted $\Aut^0 (W_\Gamma)$.

\begin{definition*}
$\Aut^0(W_\Gamma)$ consists of all automorphisms of $W_\Gamma$ that map each vertex to a conjugate of itself. 
\end{definition*}

$\Aut^0(W_\Gamma) \lhd \Aut(W_\Gamma)$ is generated by the set of all partial conjugations or the set of all elementary partial conjugations \cite{Muhlherr,Lauthesis}. 

The quotient of $\Aut^0(W_\Gamma)$ by the inner automorphisms gives a subgroup $\Out^0(W_\Gamma)$ of the full outer automorphism group. This quotient splits, and $\Out^0(W_\Gamma)$ is isomorphic to a subgroup of the full automorphism group. In fact, a full decomposition of the automorphism group was given in \cite{GPR}:

\begin{theorem*}[Gutierrez-Piggott-Ruane]
\[\Aut(W_\Gamma) = \underbrace{\left( \Inn \left(W_\Gamma \right) \rtimes \Out^0 \left(W_\Gamma \right) \right)}_{\Aut^0 \left(W_\Gamma \right)} \rtimes \Aut^1 \left(W_\Gamma \right)\] 
\end{theorem*}

Now $\Inn(W_\Gamma) \cong W_\Gamma / Z(W_\Gamma)$, and the center of a RACG is the subgroup generated by the vertices of $\Gamma$ connected to all other vertices \cite{GPR}. $W_\Gamma$ then splits as $W_{\Gamma'} \times Z(W_\Gamma)$, where $\Gamma'$ is  the induced graph in $\Gamma$ of the non-central vertices. Thus, $\Inn(W_\Gamma) \cong W_{\Gamma'}$ is a RACG itself.

Additionally, for a RACG $W_\Gamma$, $\Aut^1(W_\Gamma)$ is a subgroup of $\GL(n,2)$, and so is a finite group \cite{GPR}. So, both $\Aut^1(W_\Gamma)$ and $\Inn(W_\Gamma)$ have well-understood large scale geometry. Therefore, studying the geometry of $\Aut^0(W_\Gamma)$, or even $\Aut(W_\Gamma)$, relies on understanding the geometry of $\Out^0(W_\Gamma)$. 

Since $\Aut^0(W_\Gamma)$ and $\Out^0(W_\Gamma)$ are generated by involutions (the partial conjugations), it is a natural question to ask:

\begin{question*}
For a given RACG $W_{\Gamma}$, are $\Aut^0(W_\Gamma)$ or $\Out^0(W_\Gamma)$ themselves RACGs or even just $\cat(0)$ groups?
\end{question*}

To answer this, we need not just a generating set but a full finite presentation for $\Aut^0(W_\Gamma)$ and $\Out^0(W_\Gamma)$ and preferably a geometric model for each to act upon. A full presentation for $\Aut^0(W_\Gamma)$ is given in both \cite{Lauthesis, Muhlherr}, and McCullough-Miller space will give one such potential geometric model for the simpler case of $\Out(W_n)$ \cite{Piggott}. 

For $W_n$, there are no transvections and $\Aut^1(W_n)$ consists of only the graph symmetries and so is isomorphic to $\Sigma_n$, the symmetric group on $n$ letters. Since $W_n$ has trivial center, $\Inn(W_n) \cong W_n$. Thus in the case of $W_n$, we have the decomposition:

\begin{cor*}
\begin{align*}
\Aut(W_n) &= \underbrace{\left( W_n \rtimes \Out^0 \left(W_n \right) \right)}_{\Aut^0 \left(W_n \right)} \rtimes \Sigma_n \\
\Out (W_n) &= \Out^0 \left(W_n \right) \rtimes \Sigma_n 
\end{align*}
\end{cor*}

\begin{remark*}
When we write $x_{i,D} \in \Out^0(W_n)$, we can think of $\Out^0(W_n)$ as either a subgroup of $\Aut(W_n)$, in which case $x_{i,D}$ is a single automorphism, or else as a subgroup of $\Out(W_n)$, in which case $x_{i,D}$ is an equivalence class of automorphisms that differ by inner automorphisms. In the former case, both the acting letter $i$ and the domain $D$ are uniquely determined by the group element $x_{i,D}$. In the latter case, this is almost true. The acting letter $i$ is determined, but there are exactly two domains that result in the same outer automorphism class, namely $x_{i,D} = x_{i,D^{c}}$, where $D^{c} = [n] \setminus \{D \cup \{i\}\}$. If we need to pick a unique representative for $x_{i,D}$, we follow \cite{Piggott} and choose the $D$ that does \emph{not} contain the smallest possible index (which is usually 1, unless 1 is the acting letter, in which case it is 2). 
\end{remark*}

What about the geometry of $\Out^0(W_n)$? While $\Aut(W_3)$ is known to be $\cat(0)$ \cite{PRW} and $\Out^0(W_3) \cong W_3$, for $n \geq 4$, it was open as to whether or not $\Aut^0(W_n)$ or $\Out^0(W_n)$ is a right-angled Coxeter group or even a $\cat(0)$ group. It is now known that $\Out^0(W_n)$ is not a right-angled Coxeter group \cite{mythesis}.

For each of the groups $G = \Out^0(W_n)$ or $\Out(W_n)$, we might ask the following questions:
\begin{enumerate}
%\item Is $G$ a right-angled Coxeter group?
\item Is $G$ a $\cat(0)$ group?
\item Is there an accurate geometric model for $G$, i.e., a geodesic metric space $X$ such that $\Isom(X) \cong G$? 

\end{enumerate}

Piggott \cite{Piggott} proved that McCullough-Miller space is an accurate combinatorial and topological model for $\Out(W_n)$, although we show in Section  \ref{sec:metmm4} that it cannot be promoted to a true geometric model for either $\Out(W_n)$ or $\Out^0(W_n)$.

In particular, we prove the following main theorem:

\begin{reptheorem}{thm:noout0n}
There does not exist an $\Out^0(W_n)$-equivariant (or $\Out(W_n)$-equivariant) piecewise Euclidean (or piecewise hyperbolic) $\cat(0)$ $\LP\cat(-1)\RP$ metric on $\K_n$ for $n \geq 4$. 
\end{reptheorem}

\section{Hypertrees}\label{sec:hypertrees}

The following section is inspired by the exposition in ~\cite{Piggott}.

An accurate geometric model for $\Out^0(W_n)$ is given by McCullough-Miller space, which was originally defined using a simplicial complex associated to labeled bipartite trees \cite{MM}. However, an equivalent definition of the space is derived through a complex of labeled hypertrees \cite{McCammond}.

The connection between hypertrees and $\Out^0(W_n)$ is encapsulated in the following main theorem of this section.

\begin{reptheorem}{thm:comhyp}
Let $x_{i_1,D_1}, x_{i_2,D_2}, \dotsc, x_{i_p,D_p}$ be partial conjugations in $\Out^0(W_n) \leq \Aut^0(W_n)$. Then there exists a hypertree $\Theta \in \HTc_n$ that carries all of the \[x_{i_1,D_1}, x_{i_2,D_2}, \dotsc, x_{i_p,D_p}\] if and only if  they pairwise commute. 
\end{reptheorem}

First, we must define the relevant concepts. 

\begin{definition}
A \emph{hypergraph} $\Gamma$ is an ordered pair $(V_\Gamma , E_\Gamma )$ consisting of a set of \emph{vertices} $V_\Gamma$ and a set of \emph{hyperedges} $E_\Gamma$, where for each $e \in E_\Gamma$, $e \subseteq V_\Gamma$ and $|e| \geq 2$. Often we will label the vertices which leads to a \emph{labeled hypergraph}, and we say that $\Gamma$ is a \emph{(labeled) hypergraph on $V_\Gamma$}. A hypergraph in which every edge contains exactly two vertices is a \emph{(simple) graph}. 

We consider two equivalences on the class of hypergraphs. First, two hypergraphs $\Gamma$ and $\Gamma'$ are \emph{isomorphic as unlabeled hypergraphs} if there exists a bijection $f: V_\Gamma \to V_{\Gamma'}$ such that for each subset $S \subseteq V_\Gamma$, $f(S) \in E_{\Gamma'}$ if and only if $S \in E_{\Gamma}$. $f$ is then called a \emph{hypergraph isomorphism}. Second, two hypergraphs $\Gamma$ and $\Gamma'$ are \emph{isomorphic as labeled hypergraphs} if $V_{\Gamma} = V_{\Gamma'}$ and the identity map $V_{\Gamma} \to V_{\Gamma}$ is a hypergraph isomorphism. Unless stated otherwise, labeled hypergraphs will be considered up to labeled hypergraph isomorphism. 

A \emph{simple walk} from $v$ to $v'$ in $\Gamma$ is a sequence of alternating hypervertices and hyperedges $v=v_0 \overset{e_1}{\rightarrow} v_1 \overset{e_2}{\rightarrow}  \dotsb \overset{e_p}{\rightarrow} v_p = v'$ where $\{v_i , v_{i+1}\} \subseteq e_{i+1}$ for all $0 \leq i \leq p-1$, $v_{i} \neq v_{j}$ for all $0 \leq i\neq j \leq p$, and $e_{i} \neq e_{j}$ for all $1 \leq i\neq j \leq p$.

A \emph{hypertree} is a hypergraph $\Gamma$ where for all $v,w \in V_\Gamma$, there exists a unique simple walk from $v$ to $w$ in $\Gamma$. A hypertree which is also a graph is a \emph{tree}. 
\end{definition}

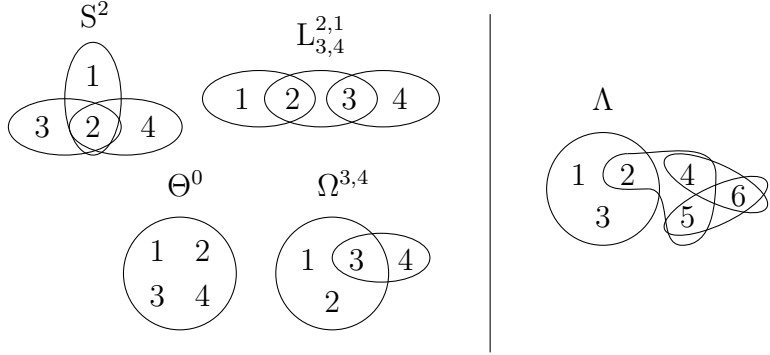
\begin{figure}
\begin{center}
\begin{tikzpicture}[scale=1.5]

\node (T) at (1.8, -0.5) {$\Theta^0$};

\draw ($(T)+(-0.05,-0.8)$) circle [radius=0.5];
\node at ($(T)+(-0.25,-0.6)$) {$1$};
\node [color=black] at ($(T)+(0.15,-0.6)$) {$2$};
\node [color=black] at ($(T)+(-0.25,-1.0)$) {$3$};
\node [color=black] at ($(T)+(0.15,-1.0)$) {$4$};

\node (L) at (3, 0.8) {$\mathrm{L}^{2,1}_{3,4}$};

\draw ($(L)+(-0.55,-0.55)$) circle[color=black, x radius=0.5, y radius=0.25, rotate = 0];
\draw ($(L)+(0,-0.55)$) circle[color=black, x radius=0.5, y radius=0.25, rotate = 0];
\draw ($(L)+(0.55,-0.55)$) circle[color=black, x radius=0.5, y radius=0.25, rotate = 0];
\node [color=black] at ($(L)+(-0.7,-0.55)$) {$1$};
\node [color=black] at ($(L)+(-0.25,-0.55)$) {$2$};
\node [color=black] at ($(L)+(0.25,-0.55)$) {$3$};
\node [color=black] at ($(L)+(0.7,-0.55)$) {$4$};

\node (S) at (1, 1) {$\mathrm{S}^{2}$};

\draw ($(S)+(-0.27,-1)$) circle[color=black, x radius=0.5, y radius=0.25, rotate = 0];
\draw ($(S)+(0.27,-1)$) circle[color=black, x radius=0.5, y radius=0.25, rotate = 0];
\draw ($(S)+(-0.02,-0.75)$) circle[color=black, x radius=0.5, y radius=0.25, rotate = 90];
\node [color=black] at ($(S)+(-0.02,-1)$) {$2$};
\node [color=black] at ($(S)+(-0.47,-1)$) {$3$};
\node [color=black] at ($(S)+(0.47,-1)$) {$4$};
\node [color=black] at ($(S)+(-0.02,-0.55)$) {$1$};

\node (O) at (3.2, -0.5) {$\Omega^{3,4}$};

\draw ($(O)+(-0.1,-0.8)$) circle [radius=0.5];
\draw ($(O)+(0.333,-0.66)$) circle[color=black, x radius=0.433, y radius=0.2165, rotate = 0];
\node [color=black] at ($(O)+(0.1165,-0.675)$) {$3$};
\node [color=black] at ($(O)+(-0.3165,-0.675)$) {$1$};
\node [color=black] at ($(O)+(-0.1,-1.05)$) {$2$};
\node [color=black] at ($(O)+(0.5495,-0.675)$) {$4$};

\draw (4.5,1) -- (4.5,-2);

\node (N) at (5.5, 0.25) {$\Lambda$};

\draw ($(N)+(0,-0.8)$) circle [radius=0.5];
\draw ($(N)+(1,-0.75)$) circle[color=black, x radius=0.5, y radius=0.17, rotate = -25];
\draw ($(N)+(1,-0.95)$) circle[color=black, x radius=0.5, y radius=0.17, rotate = 25];

\draw ($(N)+(0,-0.675)$) to[out=-90,in=180] ($(N)+(0.4,-0.8)$);
\draw ($(N)+(0.4,-0.8)$) to[out=0,in=180] ($(N)+(0.7,-1.3)$);
\draw ($(N)+(0.7,-1.3)$) to[out=0,in=-90] ($(N)+(1,-0.675)$);
\draw ($(N)+(1,-0.675)$) to[out=90,in=0] ($(N)+(0.2165,-0.4875)$);
\draw ($(N)+(0.2165,-0.4875)$) to[out=180,in=90] ($(N)+(0,-0.675)$);

\node [color=black] at ($(N)+(0.2165,-0.675)$) {$2$};
\node [color=black] at ($(N)+(-0.2165,-0.675)$) {$1$};
\node [color=black] at ($(N)+(0,-1.05)$) {$3$};
\node [color=black] at ($(N)+(0.75,-0.675)$) {$4$};
\node [color=black] at ($(N)+(0.75,-1.05)$) {$5$};
\node [color=black] at ($(N)+(1.2,-0.8625)$) {$6$};

\end{tikzpicture}
\end{center}
\caption[Examples of Hypergraphs]{Examples of hypergraphs: $\Theta^0$, $\mathrm{S}^2$, $\mathrm{L}^{2,1}_{3,4}$, and $\Omega^{3,4}$ are hypertrees. $\mathrm{S}^2$ and $\mathrm{L}^{2,1}_{3,4}$ are trees. $\Lambda$ is a hypergraph but not a hypertree, since both $4 \rightarrow 6 \rightarrow 5$ and $4 \rightarrow 5$ are simple walks in $\Lambda$.}
\end{figure}

\begin{remark*}[\cite{Piggott}]
The set of hypertrees on a set $S$ is in one-to-one correspondence with the set of bipartite labeled trees whose labeled vertices are in bijection with $S$.  
\end{remark*}

\begin{definition}
For each positive integer $n$, let $[n] := \{1,2,\dotsc,n\}$. Consider $\mathcal{HT}_{n}$, defined to be the set of hypertrees on $[n]$ up to labeled hypergraph isomorphism.

Given hypertrees $\Theta, \Theta' \in \HTc_n$, we say that $\Theta'$ is obtained from $\Theta$ by \emph{a single fold} if there exists distinct hyperedges $e,e' \in E_{\Theta}$ such that $e \cap e' \neq \varnothing$ and 
\[
E_{\Theta'} = \left( E_{\Theta} \setminus \{e,e'\} \right) \cup \{e \cup e'\},
\]
i.e., $E_{\Theta'}$ is the result of replacing $e$ and $e'$ in $E_{\Theta}$ by their union (which is still a hyperedge). Since $e$ and $e'$ are required to intersect, folding a hypertree results in a hypertree.  For each pair $\Theta, \Lambda \in \HTc_n$, we write $\Theta \leq \Lambda$ and say that $\Theta$ is a result of \emph{folding} $\Lambda$ if $\Theta$ may be obtained from $\Lambda$ by a (possibly empty) sequence of folds. Then $\left(\HTc_{n},\leq\right)$ is a partially ordered set called the \emph{hypertree poset of rank $n$}. We will often abuse notation and refer to this partially ordered set by $\HTc_n$. 
\end{definition}

\begin{figure}
\begin{center}
\begin{tikzpicture}[scale=1.5]

\node (T) at (0, 0) {$\Theta^0$};

\draw ($(T)+(-0.05,-0.8)$) circle [radius=0.5];
\node at ($(T)+(-0.25,-0.6)$) {$1$};
\node [color=black] at ($(T)+(0.15,-0.6)$) {$2$};
\node [color=black] at ($(T)+(-0.25,-1.0)$) {$3$};
\node [color=black] at ($(T)+(0.15,-1.0)$) {$4$};

\node (O) at (-1, 1.5) {$\Omega^{3,4}$};

\draw ($(O)+(-0.1,-0.8)$) circle [radius=0.5];
\draw ($(O)+(0.333,-0.66)$) circle[color=black, x radius=0.433, y radius=0.2165, rotate = 0];
\node [color=black] at ($(O)+(0.1165,-0.675)$) {$3$};
\node [color=black] at ($(O)+(-0.3165,-0.675)$) {$1$};
\node [color=black] at ($(O)+(-0.1,-1.05)$) {$2$};
\node [color=black] at ($(O)+(0.5495,-0.675)$) {$4$};

\node (O2) at (1, 1.5) {$\Omega^{1,2}$};

\draw ($(O2)+(-0.1,-0.8)$) circle [radius=0.5];
\draw ($(O2)+(0.333,-0.66)$) circle[color=black, x radius=0.433, y radius=0.2165, rotate = 0];
\node [color=black] at ($(O2)+(0.1165,-0.675)$) {$1$};
\node [color=black] at ($(O2)+(-0.3165,-0.675)$) {$3$};
\node [color=black] at ($(O2)+(-0.1,-1.05)$) {$4$};
\node [color=black] at ($(O2)+(0.5495,-0.675)$) {$2$};

\node (L) at (-2.3, 2.7) {$\mathrm{L}^{2,1}_{3,4}$};

\draw ($(L)+(-0.55,-0.55)$) circle[color=black, x radius=0.5, y radius=0.25, rotate = 0];
\draw ($(L)+(0,-0.55)$) circle[color=black, x radius=0.5, y radius=0.25, rotate = 0];
\draw ($(L)+(0.55,-0.55)$) circle[color=black, x radius=0.5, y radius=0.25, rotate = 0];
\node [color=black] at ($(L)+(-0.7,-0.55)$) {$1$};
\node [color=black] at ($(L)+(-0.25,-0.55)$) {$2$};
\node [color=black] at ($(L)+(0.25,-0.55)$) {$3$};
\node [color=black] at ($(L)+(0.7,-0.55)$) {$4$};

\node (L2) at (2.3, 2.7) {$\mathrm{L}^{1,2}_{3,4}$};

\draw ($(L2)+(-0.55,-0.55)$) circle[color=black, x radius=0.5, y radius=0.25, rotate = 0];
\draw ($(L2)+(0,-0.55)$) circle[color=black, x radius=0.5, y radius=0.25, rotate = 0];
\draw ($(L2)+(0.55,-0.55)$) circle[color=black, x radius=0.5, y radius=0.25, rotate = 0];
\node [color=black] at ($(L2)+(-0.7,-0.55)$) {$2$};
\node [color=black] at ($(L2)+(-0.25,-0.55)$) {$1$};
\node [color=black] at ($(L2)+(0.25,-0.55)$) {$3$};
\node [color=black] at ($(L2)+(0.7,-0.55)$) {$4$};

\node (S) at (0, 3.2) {$\mathrm{S}^{3}$};

\draw ($(S)+(-0.27,-1)$) circle[color=black, x radius=0.5, y radius=0.25, rotate = 0];
\draw ($(S)+(0.27,-1)$) circle[color=black, x radius=0.5, y radius=0.25, rotate = 0];
\draw ($(S)+(-0.02,-0.75)$) circle[color=black, x radius=0.5, y radius=0.25, rotate = 90];
\node [color=black] at ($(S)+(-0.02,-1)$) {$3$};
\node [color=black] at ($(S)+(-0.47,-1)$) {$2$};
\node [color=black] at ($(S)+(0.47,-1)$) {$4$};
\node [color=black] at ($(S)+(-0.02,-0.55)$) {$1$};

\draw ($(L)+(0,-0.85)$) -- ($(O)+(-0.4,-0.35)$);
\draw ($(L2)+(-1.1,-0.55)$) -- ($(O)+(0.4,-0.35)$);
\draw ($(S)+(0,-1.28)$) -- ($(O)+(0.3,-0.3)$);
\draw ($(L2)+(0,-0.85)$) -- ($(O2)+(0.3,-0.3)$);
\draw ($(O)+(0.2,-1.3)$) -- ($(T)+(-0.3,-0.3)$);
\draw ($(O2)+(-0.3,-1.35)$) -- ($(T)+(0.3,-0.35)$);

\end{tikzpicture}
\end{center}
\caption[The Partial Order on Hypertrees]{The lines represent folding relations on hypertrees in $\HTc_4$. So $\Theta^0 \leq \Omega^{3,4} \leq \mathrm{L}^{2,1}_{3,4}$, $\mathrm{ S}^3$, $\mathrm{L}^{1,2}_{3,4}$, while $\Theta^0 \leq \Omega^{1,2} \leq \mathrm{L}^{1,2}_{3,4}$.}
\label{fig:htc4}
\end{figure}
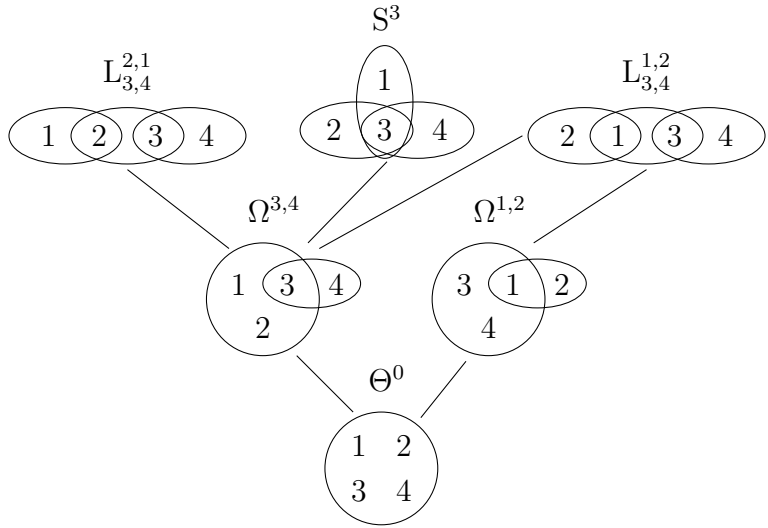

\begin{definition}
The simplicial realization of $\left(\HTc_n, \leq \right)$ is the \emph{hypertree complex of rank $n$}, $\HT_n$. This means that $\HT_n$ is a simplicial complex whose vertices are in bijective correspondence with the set of hypertrees in $\HTc_n$ and where $\Theta_1, \Theta_2, \dotsc , \Theta_k$ span a $k$-simplex in $\HT_n$ if and only if (up to reordering) $\Theta_1 \leq \Theta_2 \leq \dotsb \leq \Theta_k$ in $\HTc_n$. Since maximal chains in $\HTc_n$ involve folding trees a single fold at a time, the dimension of $\HT_n$ is $n-2$. 
\end{definition}

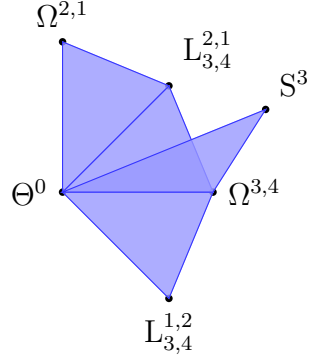
\begin{figure}
\begin{center}
\begin{tikzpicture}[scale=2]

\node (T) [circle,fill,inner sep=1pt,label=180:$\Theta^0$] at (0, 0) {};
\node (L) [circle,fill,inner sep=1pt,label=45:$\mathrm{L}^{2,1}_{3,4}$] at (0.707, 0.707) {};
\node (O) [circle,fill,inner sep=1pt,label=0:$\Omega^{3,4}$] at (1, 0) {};
\node (O2) [circle,fill,inner sep=1pt,label=90:$\Omega^{2,1}$] at (0, 1) {};
\node (L2) [circle,fill,inner sep=1pt,label=270:$\mathrm{L}^{1,2}_{3,4}$] at (0.707, -0.707) {};

\node (S) [circle,fill,inner sep=1pt,label=45:$\mathrm{S}^{3}$] at (1.35, 0.55) {};

\draw [fill,opacity = 0.8, color=blue!40!white] ($(T)$) -- ($(O)$) -- ($(L)$) -- ($(T)$) -- cycle;
\draw [color=blue!80!white] ($(T)$) -- ($(O)$) -- ($(L)$) -- ($(T)$) -- cycle;

\draw [fill,opacity = 0.8, color=blue!40!white] ($(T)$) -- ($(O)$) -- ($(L2)$) -- ($(T)$) -- cycle;
\draw [color=blue!80!white] ($(T)$) -- ($(O)$) -- ($(L2)$) -- ($(T)$) -- cycle;

\draw [fill,opacity = 0.8, color=blue!40!white] ($(T)$) -- ($(L)$) -- ($(O2)$) -- ($(T)$) -- cycle;
\draw [color=blue!80!white] ($(T)$) -- ($(L)$) -- ($(O2)$) -- ($(T)$) -- cycle;

\draw [fill,opacity = 0.8, color=blue!40!white] ($(T)$) -- ($(O)$) -- ($(S)$) -- ($(T)$) -- cycle;
\draw [color=blue!80!white] ($(T)$) -- ($(O)$) -- ($(S)$) -- ($(T)$) -- cycle;

\end{tikzpicture}
\end{center}
\caption[The Hypertree Complex, $\HT_4$]{A portion of the hypertree complex, $\HT_4$.}
\label{fig:ht4}
\end{figure}

\begin{remark*}
For $n=4$, $|\HTc_4| = 29$ and the height of $\HTc_4$ is 3. Thus, $\HT_4$ is a simplicial 2-complex. 
\end{remark*}

Now $\Sigma_n$ acts on $\HTc_n$ in an obvious way: Each permutation of $[n]$ just permutes the labels of the hypertrees, which preserves the partial order, and so is an order automorphism of $\HTc_n$. This action by order automorphisms of $\left(\HTc_n, \leq \right)$ naturally extends to an action by simplicial automorphisms on $\HT_n$ \cite{Piggott}. One might wonder: Are there any other hidden automorphisms of either $\HTc_n$ or $\HT_n$? It turns out the answer is ``no''. 

\begin{theorem}[Piggott \cite{Piggott}]
For all integers $n \geq 3$, 
\[
\Aut(\HT_n) \cong \Aut(\HTc_n) \cong \Sigma_n,
\]
where $\Aut(\HT_n)$ is the set of simplicial automorphisms of $\HT_n$, $\Aut(\HTc_n)$ is the set of order isomorphisms of $\HTc_n$, and  $\Sigma_n$ is the symmetric group on $n$ letters. 
\end{theorem}

Thus, $\HT_n$ provides an accurate (topological) model and $\HTc_n$ provides an accurate (combinatorial) model for $\Sigma_n$. If we endowed $\HT_n$ with any $\Sigma_n$-equivariant metric, for instance a piecewise Euclidean one with equilateral triangles, then $\Sigma_n$ would act by isometries and so $\HT_n$ would be an accurate geometric model for $\Sigma_n$ as well. 

\begin{definition}\label{def:hypclass}
A hypertree $\Theta \in \HTc_n$ has between one and $n-1$ hyperedges, and the \emph{height} of $\Theta$ is defined to be one less than its number of hyperedges. Notice that hypertrees of height $n-2$ are actually trees.

We note a few special classes of hypertrees:
\begin{enumerate}
\item There is a unique hypertree of height zero, denoted $\Theta_n^0$.
\item $\mathcal{S}_n = \{\mathrm{S}_n^j \mid j \in [n]\}$, the set of \emph{star trees}, where $\mathrm{S}_n^j$ is the hypertree of height $n-2$ (tree) whose hyperedges (edges) are exactly $\{i,j\}$ for $i \neq j$. 
\item $\mathcal{L}_n$, the set of \emph{line trees}, which are the trees (hypertrees of height $n-2$) in which exactly two vertices are leaves. 
\item $\mathcal{M}_n^1 = \{\Omega_n^{i,j} \mid i \neq j \in [n]\}$, the set of \emph{omega hypertrees}, are those hypertrees of height 1 that contain the hyperedges $\{i,j\}$ and $[n]\setminus \{j\}$. 
\end{enumerate}
\end{definition}

Two elements in one of these classes are isomorphic as unlabeled hypertrees, and so the action of $\Sigma_n$ on $\HTc_n$ acts transitively on each of these classes. Additionally, in $W_4$, this list actually exhausts all possible hypertrees. 

\begin{question*}
What does $\HT_n$ have to do with $\Out^0(W_n)$?
\end{question*}

It turns out that hypertrees encode commuting relations in $\Out^0(W_n)$.

\begin{definition}
A hypertree $\Theta \in \HTc_n$ \emph{carries} a partial conjugation $x_{i,D}$ if and only if for all $d \in D$, $j \in [n] \setminus D$, the simple walk from $d$ to $j$ visits $i$. 

A general automorphism $\alpha \in \Out^0(W_n)$ is \emph{carried} by $\Theta$ if and only if there exists partial conjugations $x_{i_1,D_1}, x_{i_2,D_2}, \dotsc, x_{i_p,D_p} \in \Out^0(W_n)$ such that $\alpha = x_{i_1,D_1} x_{i_2,D_2} \dotsb x_{i_p,D_p}$ and $x_{i_j,D_j}$ is carried by $\Theta$ for each $1 \leq j \leq p$. 
\end{definition}

For this definition, we may think of $\Out^0(W_n)$ as either a subgroup or a quotient of $\Aut^0(W_n)$.  Inner automorphisms are trivially carried by all hypertrees, since the only element of $[n] \setminus D$ is $i$. Thus, the notion of a hypertree carrying an automorphism is actually well-defined up to outer automorphism class. In particular, we can use this fact to freely switch between representatives  $x_{i,D} = x_{i,D^{c}}$ in $\Out^0(W_n)$, where $D^{c} = [n] \setminus \{D \cup \{i\}\}$. For notation, also let $\widetilde{D} = D \cup \{i\}$ and $\widetilde{D^c} = D^c \cup \{i\}$.

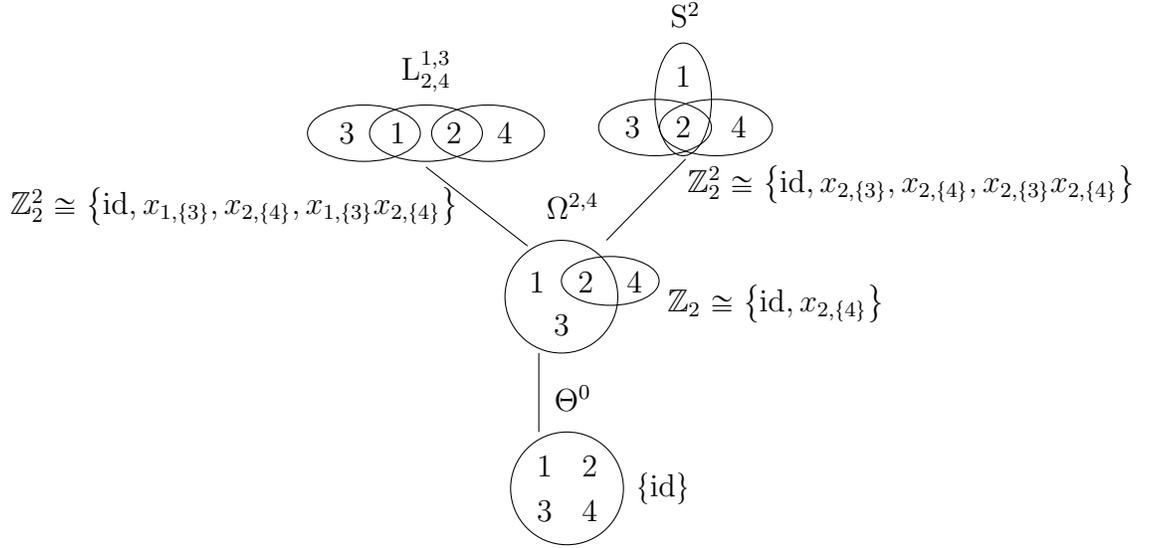
\begin{figure}
\begin{center}
\begin{tikzpicture}[scale=1.5]

\node (T) at (-1, -0.2) {$\Theta^0$};

\draw ($(T)+(-0.05,-0.8)$) circle [radius=0.5];
\node at ($(T)+(-0.25,-0.6)$) {$1$};
\node [color=black] at ($(T)+(0.15,-0.6)$) {$2$};
\node [color=black] at ($(T)+(-0.25,-1.0)$) {$3$};
\node [color=black] at ($(T)+(0.15,-1.0)$) {$4$};
\node [color=black] at ($(T)+(0.8,-0.8)$) {$ \left\{\textrm{id}\right\}$};

\node (O) at (-1, 1.5) {$\Omega^{2,4}$};

\draw ($(O)+(-0.1,-0.8)$) circle [radius=0.5];
\draw ($(O)+(0.333,-0.66)$) circle[color=black, x radius=0.433, y radius=0.2165, rotate = 0];
\node [color=black] at ($(O)+(0.1165,-0.675)$) {$2$};
\node [color=black] at ($(O)+(-0.3165,-0.675)$) {$1$};
\node [color=black] at ($(O)+(-0.1,-1.05)$) {$3$};
\node [color=black] at ($(O)+(0.5495,-0.675)$) {$4$};
\node [color=black] at ($(O)+(1.8,-0.8625)$) {$\mathbb{Z}_2 \cong \left\{\textrm{id},x_{2,\{4\}}\right\}$};

\node (L) at (-2.3, 2.7) {$\mathrm{L}^{1,3}_{2,4}$};

\draw ($(L)+(-0.55,-0.55)$) circle[color=black, x radius=0.5, y radius=0.25, rotate = 0];
\draw ($(L)+(0,-0.55)$) circle[color=black, x radius=0.5, y radius=0.25, rotate = 0];
\draw ($(L)+(0.55,-0.55)$) circle[color=black, x radius=0.5, y radius=0.25, rotate = 0];
\node [color=black] at ($(L)+(-0.7,-0.55)$) {$3$};
\node [color=black] at ($(L)+(-0.25,-0.55)$) {$1$};
\node [color=black] at ($(L)+(0.25,-0.55)$) {$2$};
\node [color=black] at ($(L)+(0.7,-0.55)$) {$4$};
\node [color=black] at ($(L)+(-1.7,-1.2)$) {$\mathbb{Z}_2^2 \cong \left\{\textrm{id},x_{1,\{3\}},x_{2,\{4\}},x_{1,\{3\}} x_{2,\{4\}}\right\}$};

\node (S) at (0, 3.2) {$\mathrm{S}^{2}$};

\draw ($(S)+(-0.27,-1)$) circle[color=black, x radius=0.5, y radius=0.25, rotate = 0];
\draw ($(S)+(0.27,-1)$) circle[color=black, x radius=0.5, y radius=0.25, rotate = 0];
\draw ($(S)+(-0.02,-0.75)$) circle[color=black, x radius=0.5, y radius=0.25, rotate = 90];
\node [color=black] at ($(S)+(-0.02,-1)$) {$2$};
\node [color=black] at ($(S)+(-0.47,-1)$) {$3$};
\node [color=black] at ($(S)+(0.47,-1)$) {$4$};
\node [color=black] at ($(S)+(-0.02,-0.55)$) {$1$};
\node [color=black] at ($(S)+(2,-1.5)$) {$\mathbb{Z}_2^2 \cong \left\{\textrm{id},x_{2,\{3\}},x_{2,\{4\}},x_{2,\{3\}} x_{2,\{4\}}\right\}$};

\draw ($(L)+(0,-0.85)$) -- ($(O)+(-0.4,-0.35)$);
\draw ($(S)+(0,-1.28)$) -- ($(O)+(0.3,-0.3)$);
\draw ($(O)+(-0.3,-1.3)$) -- ($(T)+(-0.3,-0.3)$);

\end{tikzpicture}
\end{center}
\caption[Automorphisms Carried by Hypertrees]{A portion of $\HTc_4$ and the automorphisms in $\Out^0(W_4)$ carried by each hypertree.}
\end{figure}

\begin{remark}
\label{rem:carry}
Hypertrees of height $h$ carry $2^h$ automorphisms in $\Out^0(W_n)$, including the identity automorphism, and if $\Theta \leq \Lambda$, then $\Lambda$ carries all the automorphisms that $\Theta$ does \cite{Piggott}. In fact, the $2^h$ automorphisms carried by $\Theta$ all commute and generate a $\mathbb{Z}_2^h$, which follows from Theorem \ref{thm:comhyp} below.
\end{remark}

\begin{lemma}[Gutierrez-Piggott-Ruane]
\label{lem:compc}
Let $x_{i_1,D_1}$ and $x_{i_2,D_2}$ be partial conjugations in $\Out^0(W_n) \leq \Aut^0(W_n)$. Then they commute if and only if one of the following four cases hold:
\begin{enumerate}
\item $i_1 = i_2$
\item $i_1 \neq i_2$, $i_1 \in D_2$, $i_2 \not\in D_1$, $D_1 \subseteq D_2$, i.e., $i_1 \neq i_2$, $\widetilde{D_1} \cap \widetilde{D_2^c} = \varnothing$.
\item $i_1 \neq i_2$, $i_1 \not\in D_2$, $i_2 \in D_1$, $D_2 \subseteq D_1$, i.e., $i_1 \neq i_2$, $\widetilde{D_1^c} \cap \widetilde{D_2} = \varnothing$.
\item $i_1 \neq i_2$, $i_1 \not\in D_2$, $i_2 \not\in D_1$, $D_1 \cap D_2 = \varnothing$, i.e., $i_1 \neq i_2$, $\widetilde{D_1} \cap \widetilde{D_2} = \varnothing$.
\end{enumerate}

%Equivalently, letting $\widetilde{D_j} = D_j \cup \{i_j\}$ and $\widetilde{D_j^c} = D_j^c \cup \{i_j\} = [n] \setminus D_j$:
%
%\begin{enumerate}
%\item $i_1 = i_2$
%\item $i_1 \neq i_2$, $i_2 \not\in D_1$, $\widetilde{D_1} \subseteq D_2$
%\item $i_1 \neq i_2$, $i_1 \not\in D_2$, $\widetilde{D_2} \subseteq D_1$
%\item $i_1 \neq i_2$, $\widetilde{D_2} \subseteq D_1^c $
%\end{enumerate}

\end{lemma}

\begin{proof}
This follows from Lemma 4.3 in \cite{GPR}.
\end{proof}

\begin{theorem}
\label{thm:comhyp}
Let $x_{i_1,D_1}, x_{i_2,D_2}, \dotsc, x_{i_p,D_p}$ be partial conjugations in $\Out^0(W_n) \leq \Aut^0(W_n)$. Then there exists a hypertree $\Theta \in \HTc_n$ that carries all of the \[x_{i_1,D_1}, x_{i_2,D_2}, \dotsc, x_{i_p,D_p}\] if and only if  they pairwise commute. 
\end{theorem}

\begin{proof}
One direction is Lemma 4.4 in \cite{Piggott}, and is reproduced here for convenience:

Suppose that $\Theta$ carries each of the $x_{i_j,D_j}$ for $j \in \{1,\dotsc,p\}$. If $i_k \neq i_l$, then $x_{i_k,D_k}$ and $x_{i_l,D_l}$ commute by Lemma 1.1 in \cite{MM}. Because the $\mathbb{Z}_2$ factors in $W_n$ are abelian (or just directly from the definition of partial conjugation), whenever $i_k = i_l$, then $x_{i_k,D_k}$ and $x_{i_l,D_l}$ commute. 

Conversely, suppose that $x_{i_1,D_1}, x_{i_2,D_2}, \dotsc, x_{i_p,D_p}$ pairwise commute. 

%If all the acting letters are the same, say $i$, then they are all carried by the hypertree (in fact the tree) $S_i$, which is the ``star'' tree that has $i$ in the center and a single (hyper)edge $\{i,j\}$ for each $j \neq i$. This tree carries all the partial conjugations since \emph{any} nontrivial simple walk must pass through $i$.

%So now, suppose that at least two of the partial conjugations have different acting letters. 
We will build the hypertree $\Theta$ inductively. 

Let $\Theta_1$ be the hypertree on $[n]$ that has two hyperedges: one containing $\widetilde{D_1} = D_1 \cup \{i_1\}$ and the other containing $\widetilde{D_1^c} = D_1^c \cup \{i_1\} = [n] \setminus D_1$. Any simple walk from $D_1$ to $D_1^c$ must pass through $i$, so $\Theta_1$ carries $x_{i_1,D_1}$. In fact, the only automorphisms carried by $\Theta_1$ are the identity and $x_{i_1,D_1}$.

Now inductively assume that there is a hypertree $\Theta_{k-1}$ on $[n]$ that carries $x_{i_1,D_1}, x_{i_2,D_2}, \dotsc, x_{i_{k-1},D_{k-1}}$ for $1 \leq k-1 \leq p-1$ and that $x_{i_k,D_k}$ commutes with all automorphisms carried by $\Theta_{k-1}$.  Since $\Theta_{k-1}$ is a hypertree, every hypervertex is in at least one hyperedge, and any two hyperedges are either disjoint or else intersect in exactly one hypervertex. Consider $x_{i_k,D_k}$, and denote the hyperedges of $\Theta_{k-1}$ by $E_{k-1}$. Now define $E_{k}$ to be the set of non-empty intersections between the hyperedges of $\Theta_{k-1}$ and either $D_k$ or $D_k^c$, i.e.,
\[
E_{k} := \left(\left\{E \cap \widetilde{D_k} \mid E \in E_{k-1}\right\} \cup \left \{ E \cap \widetilde{D_k^c} \mid E \in E_{k-1} \right \}\right) \setminus \varnothing,
\]
and let $\Theta_k$ be the hypergraph defined on $[n]$ with $E_{k}$ as its hyperedges. Suppose that both $E^1 = E \cap \widetilde{D_k}$ and $E^2 = E \cap \widetilde{D_k^c}$ are non-empty. We claim that $i_k \in E$ and so $E^1 \cap E^2 = \{i_k\}$:

$\Theta_{k-1}$ is a hypertree that carries at least one non-identity automorphism, so it has at least 2 hyperedges, and thus there is a neighboring hyperedge to $E$, $E'$, such that $E \cap E' = \{m\}$ for some $m \in [n]$. If $m = i_k$, then $i_k \in E$. Otherwise, suppose that $m \neq i_k$. Since $\Theta_{k-1}$ is a hypertree, $\Theta_{k-1} \setminus \{m\}$ is disconnected. Let $D_m$ be the connected component of $\Theta_{k-1} \setminus \{m\}$ that contains $E \setminus \{m\}$ and $D_m^c$ be the union of the rest of the components. Then $\Theta_{k-1}$ must carry $x_{m,D_m}$. Thus, by assumption, $x_{i_k,D_k}$ commutes with $x_{m,D_m}$. If $i_k \not\in E$, then $i_k \not\in D_m$ since $E \subseteq D_m$. Also, the non-empty element of $E^1$ can't be $i_k$, so it must be an element of $E \cap D_k$, i.e., $E \cap D_k \neq \varnothing$. But the same is true for $E^2$, $E \cap D_k^c \neq \varnothing$. By Lemma \ref{lem:compc}, this leaves only the option that $m \in D_k$ and $D_m \subseteq D_k$, which contradicts that $E \cap D_k^c \neq \varnothing$. Thus, $i_k$ must be in $E$.

Suppose that $v=v_0 \overset{e_1}{\rightarrow} v_1 \overset{e_2}{\rightarrow}  \dotsb \overset{e_p}{\rightarrow} v_p = v'$ is the unique simple walk in $\Theta_{k-1}$ from $v$ to $v'$. In $\Theta_k$, each $e_j$ is partitioned into (at most) two hyperedges, $e_j^1 = e_j \cap \widetilde{D_k}$ and $e_j^2 = e_j \cap \widetilde{D_k^c}$. If $v_{j-1}$ and $v_{j}$ are both in the same hyperedge, say $e_j^1$, then just replace $e_j$ with $e_j^1$ in the walk above. Otherwise, without loss of generality, $v_{j-1} \in e_j^1$ and $v_{j} \in e_j^2$. Since both $e_j^1$ and $e_j^2$ are non-empty, from above we know that $e_j^1 \cap e_j^2 = \{i_k\}$. Now replace $v_{j-1} \overset{e_j}{\rightarrow} v_{j}$ in the walk with $v_{j-1} \overset{e_j^1}{\rightarrow} i_k \overset{e_j^2}{\rightarrow} v_{j}$. This can only happen once in the walk since otherwise $i_k$ would be in two non-consecutive hyperedges, and a different, shorter simple walk would have been possible in $\Theta_{k-1}$. So this construction shows that $\Theta_k$ is a hypertree. It carries $x_{i_k,D_k}$ since if $v_0 \in D_k$ and $v_p \in D_k^c$, then there must be some point in the walk where the hyperedges go from the $e_j^1$ to the $e_j^2$, at which point either $i_k$ was already in the walk or else it gets inserted in the construction. 

It is immediate that $\Theta_{k-1} \leq \Theta_{k}$ since folding $E \cap \widetilde{D_k}$ and $E \cap \widetilde{D_k^c}$ into $E$ is necessary only when both are non-empty. In particular, that means that $\Theta_k$ carries all of the automorphisms that $\Theta_{k-1}$ carried. Recall that $x_{i_j,D_j}$ is carried by a hypertree $\Lambda$ if and only if $D_j$ is a union of connected components (other than the one containing the lowest index of $[n]\setminus \{i_j\}$) of $\Lambda \setminus \{i_j\}$ (\cite{Piggott}). 

Let $D_{m_1}, D_{m_2}, \dotsc, D_{m_l}$ be the connected components of $\Theta_{k} \setminus \{i_k\}$ other than the one with minimal index. These exactly correspond with the analogous connected components in $\Theta_{k-1}$ except that one is added each time a hyperedge (which had to contain $i_k$) was split. Since the number of outer automorphisms carried by a hypertree is $2^h$ (where $h$ = height = number of hyperedges minus one), this unfolding increases the height by exactly the number of edges with $e_j^1 \cap e_j^2 = \{i_k\}$. All of the $x_{i_k,D_j}$ carried by $\Theta_{k-1}$ are products of the $x_{i_k,D_{m_s}}$. So by a counting argument, all of the automorphisms carried by $\Theta_k$ are given by products of the $x_{i_j,D_j}$ (with $i_j \neq i_k$ and $1 \leq j \leq k$) and the $x_{i_k,D_{m_s}}$ (with $1 \leq s \leq l$). It suffices to prove that all of these commute with the remaining automorphisms on our list. 

Now, let $1 \leq j \leq p,\ 1 \leq s \leq l$ and consider $x_{i_j,D_j}$ and $x_{i_k,D_{m_s}}$. By construction, $D_{m_s} \subseteq D_k$ or $D_{m_s} \subseteq D_k^c$. It suffices to show that $x_{i_j,D_j}$ or $x_{i_j,D_j^c}$ commutes with $x_{i_k,D_{m_s}}$ or $x_{i_k,D_{m_s}^c}$ since $\Out^0(W_n)$ as a quotient is isomorphic to $\Out^0(W_n)$ as a subgroup of $\Aut^0(W_n)$. So without loss of generality, suppose that $D_{m_s} \subseteq D_k$, and so $\widetilde{D_{m_s}} \subseteq \widetilde{D_k}$. If $i_k = i_j$, this is trivial, so suppose not. By the definition of $x_{i_k,D_{m_s}}$, there is some component $D$ of $\Theta_{k-1} \setminus \{i_k\}$ such that $D_{m_s} = D_k \cap D$, and thus $\widetilde{D_{m_s}} = \widetilde{D_k} \cap \widetilde{D}$. $x_{i_k,D}$ is then carried by $\Theta_{k-1}$ and so commutes with $x_{i_j,D_j}$. Now $\widetilde{D_{m_s}} \cap \widetilde{D_j} = \left( \widetilde{D_k} \cap \widetilde{D_j} \right) \cap \left( \widetilde{D} \cap \widetilde{D_j} \right)$, and if this is empty, we are done. Similarly, $\widetilde{D_{m_s}} \cap \widetilde{D_j^c} = \left( \widetilde{D_k} \cap \widetilde{D_j^c} \right) \cap \left( \widetilde{D} \cap \widetilde{D_j^c} \right)$, and if this is empty, we are done. Otherwise, all four intersections must be non-empty. But by Lemma \ref{lem:compc}, this forces $\widetilde{D_k^c} \cap \widetilde{D_j} = \varnothing$ and $\widetilde{D^c} \cap \widetilde{D_j} = \varnothing$, i.e., $i_k \not\in D_j,\ i_j \in D_k,\ D_j \subseteq D_k$ and $i_k \not\in D_j,\ i_j \in D,\ D_j \subseteq D$. Thus, $i_k \not\in D_j,\ i_j \in D_{m_s},\ D_j \subseteq D_{m_s}$, and so we are done again. Thus, every automorphism on our list commutes with every automorphism carried by $\Theta_k$. This completes the induction. 

\end{proof}

In fact, examining the proof of Theorem \ref{thm:comhyp}, we actually proved a stronger corollary. 

\begin{cor}\label{cor:comhyp}
Given a hypertree $\Theta \in \HTc_n$ and a partial conjugation $x_{i,D}$, then there exists an unfolding of $\Theta$ to a hypertree $\Lambda \geq \Theta$ that carries $x_{i,D}$ if and only if $x_{i,D}$ commutes with every automorphism carried by $\Theta$. 
\end{cor}

Now that we know how the hypertree complex encodes the commuting information of $\Out^0(W_n)$, we can use this to build a complex that $\Out^0(W_n)$ can act on.

\section{McCullough-Miller Space}\label{sec:mmspace}

McCullough and Miller originally \cite{MM} defined their complex using labeled bipartite trees, but McCammond and Meier \cite{McCammond} showed an equivalent way to define the space using $\HTc_n$. Piggott \cite{Piggott} then characterized the automorphism groups of these spaces. 

McCullough-Miller space is constructed by taking a copy of $\HTc_n$ for each element of $\Out^0(W_n)$ and then gluing these copies together according to the hypertree carrying relation. 

\begin{definition}\label{def:MM}
First, define an equivalence relation $\sim$ on $\Out^0(W_n) \times \HTc_n$ as follows: $(\alpha,\Theta) \sim (\beta, \Lambda)$ if and only if $\Theta = \Lambda$ and $\alpha^{-1} \beta$ is carried by $\Theta$. Write $[\alpha,\Theta]$ for the $\sim$-equivalence class of $(\alpha,\Theta)$ and let $\Kc_n$ be the set of $\sim$-equivalence classes. 

Now, define a partial order $\leq$ on $\Kc_n$: $[\alpha,\Theta] \leq [\beta,\Lambda]$ if and only if $\Lambda$ folds to $\Theta$ and $\alpha^{-1}\beta$ is carried by $\Lambda$, i.e., $\Theta \leq \Lambda$ in $\HTc_n$ and $[\alpha,\Lambda] = [\beta,\Lambda]$. 

\emph{McCullough-Miller space} $\K_n$ is the simplicial realization of $(\Kc_n,\leq)$. We will often abuse notation and have $[\alpha,\Theta]$ refer to both its equivalence class in $\Kc_n$ as well as its corresponding vertex in $\K_n$. 
\end{definition}

\begin{remark}
For a hypertree $\Theta$ of height $h$ in $\HTc_n$, $\Theta$ carries $2^h$ automorphisms, and so $[\alpha,\Theta]$ will be glued to $2^h - 1$ other copies of $\Theta$. In particular, $\left[\alpha,\Theta_n^0\right]$ is a singleton, is not glued to any other element, and $\left[\alpha,\Theta_n^0\right] \leq [\beta,\Lambda]$ if and only if $\left[\alpha,\Lambda\right] = \left[\beta,\Lambda\right]$. These are called \emph{nuclear vertices} of $\K_n$. So $\K_n$ consists of partially glued copies of $\HT_n$ indexed by $\Out^0(W_n)$. 
\end{remark}

Recall that $\Sigma_n$ acts on $\HTc_n$ by permuting labels, and that $\Out(W_n) \cong \Out^0(W_n) \rtimes \Sigma_n$. So any $\alpha \in \Out(W_n)$ has a unique representative $\phi \sigma$, where $\phi \in \Out^0(W_n)$, $\sigma \in \Sigma_n$, and $\alpha = \phi \sigma$ in $\Out(W_n)$. 

\begin{definition}
$\Out(W_n)$ acts on $\Out^0(W_n) \times \HTc_n$ by:
\[
\phi \sigma \cdot \left(\alpha, \Theta \right) = \left(\phi (\sigma \alpha \sigma^{-1}), \sigma \Theta \right)
\]
Since $\Out^0(W_n) \trianglelefteq \Out(W_n)$, $(\sigma \alpha \sigma^{-1}) \in \Out^0(W_n)$, and so $\phi \sigma \alpha \sigma^{-1} \in \Out^0(W_n)$. The action of $\sigma$ on $\Theta$ is by permuting the labels. 

This action of $\Out(W_n)$ preserves $\sim$ as well as the partial order $\leq$. Thus, this descends to an action of $\Out(W_n)$ on $\Kc_n$ by order automorphisms as well as $\K_n$ by simplicial automorphisms \cite{Piggott}. 
\end{definition}

\begin{example*}
Let $(1\ 2) \in \Sigma_4$ be the transposition that exchanges 1 and 2, and let $(1 - 2 - 3 - 4)$ be the line (hyper)tree that contains the edges $\{1,2\},\ \{2,3\},\ \{3,4\}$.
\begin{align*}
&\phantom{= \ } \left(x_{1,\{3\}}, (1\ 2)\right) \cdot \left[x_{2,\{4\}}, (1 - 2 - 3 - 4)\right] \\
&= \left[x_{1,\{3\}} \left((1\ 2) x_{2,\{4\}} (1\ 2)^{-1} \right), (1\ 2) \cdot (1 - 2 - 3 - 4) \right] \\
&= \left[x_{1,\{3\}} x_{1,\{4\}}, (2 - 1 - 3 - 4) \right] \\
&= \left[x_{1,\{3,4\}}, (2 - 1 - 3 - 4) \right] \\
\end{align*}
\end{example*}

As with $\HT_n$, this action induces an injective map from $\Out(W_n)$ into both $\Aut(\Kc_n,\leq)$ and $\Aut(K_n)$, and one might wonder whether or not there any other other hidden symmetries in these spaces. The answer is once again in the negative, and so these spaces serve as accurate combinatorial and topological models for $\Out(W_n)$. 

\begin{theorem}[Piggott \cite{Piggott}, Thm 1.1]
For $n \geq 4$, 
\[
\Aut(\Kc_n,\leq) \cong \Aut(K_n) \cong \Out(W_n).
\] 
\end{theorem}

\begin{remark}
As in the case of $\HT_n$ and $\Sigma_n$, this shows that $\Kc_n$ is an accurate combinatorial model and $\K_n$ is an accurate topological or simplicial model for $\Out(W_n)$. In fact, $\Out(W_n)$ acts on $\K_n$ properly discontinuously and co-compactly by simplicial automorphisms (\cite{Piggott}), but $\K_n$ has no \emph{a priori} metric on it. To be an accurate \emph{geometric} model, we will need to endow $\K_n$ with a metric to turn it into a geodesic metric space such that the action of $\Out(W_n)$ is by isometries. Then $\K_n$ will be quasi-isometric to $\Out(W_n)$ (and also its finite index subgroup, $\Out^0(W_n)$), and they will have the same large-scale geometry. There are many ways to do this, such as assigning the piecewise Euclidean metric with equilateral triangles to $\K_n$.

However, this metric does not turn $\K_n$ into a $\cat(0)$ space. If we wish to use this space to show that $\Out(W_n)$ is a $\cat(0)$ group, then we will need to pick a different metric. The metric will need to be $\cat(0)$ as well as equivariant with respect to the $\Out(W_n)$ or $\Out^0(W_n)$ action on $\K_n$. As we show in Section \ref{sec:metmm4}, no such (piecewise $M_{\kappa}$) metric turns out to exist. 
\end{remark}

%Recall that $\Out(W_4) \cong \Out^0(W_4) \rtimes \Sigma_4$ acts on $\K_4$ by
%\[
%\phi \sigma \cdot \left[\alpha, \Theta \right] = \left[\phi (\sigma \alpha \sigma^{-1}), \sigma \Theta \right].
%\]

Now, let $[\alpha, \Theta] \in K_n$ and suppose that $[\beta, \Theta']$ is another point where $\Theta$ and $\Theta'$ are isomorphic as unlabeled hypertrees. Since $\Theta'$ differs from $\Theta$ only in its labeling, there is a permutation $\sigma \in \Sigma_n$ such that $\sigma \cdot \Theta = \Theta'$ \cite{Piggott}. Since $\Out^0(W_n) \trianglelefteq \Out(W_n)$, $\sigma \alpha^{-1} \sigma^{-1} \in \Out^0(W_n)$, and thus $\phi = \beta \sigma \alpha^{-1} \sigma^{-1} \in \Out^0(W_n)$. Then we have that 

\begin{align*}
\phi \sigma \cdot \left[\alpha, \Theta \right] &= \left[\phi (\sigma \alpha \sigma^{-1}), \sigma \Theta \right] \\ 
&= \left[\beta \sigma \alpha^{-1} \sigma^{-1} (\sigma \alpha \sigma^{-1}), \Theta' \right] \\
&= \left[\beta, \Theta' \right].
\end{align*}

Thus, $\Out(W_n)$ acts transitively on the subsets of $\Kc_n$ where the $\Out^0(W_n)$ labels can be anything and the unlabeled hypertree isomorphism classes are preserved. Since the action of $\Sigma_n$ on $\HTc_n$ only permutes labels, it preserves unlabeled isomorphisms classes, and so the full action of $\Out(W_n)$ on $\Kc_n$ must as well. Thus, the quotient of $\Kc_n$ by $\Out(W_n)$ consists of one simplex for each unlabeled isomorphism class in $\HTc_n$, glued along common edges.   

$\Out^0(W_n)$ acts transitively on the labels of $\Kc_n$ but doesn't change the hypertree. Thus, the quotient of $\Kc_n$ by $\Out^0(W_n)$ is the full hypertree complex $\HTc_n$. 

As noted in Definition \ref{def:hypclass}, the unlabeled isomorphism classes in $\HTc_4$ are precisely $\left\{\Theta_4^0\right\}, \mathcal{S}_4,\ \mathcal{L}_4,\ \mathcal{M}_4^1$ \cite{Piggott}. When we are only concerned with $n=4$, we will drop the subscripts and use a more descriptive notation.

\begin{notation}\label{not:ht4}
We will denote the hypertrees in $\HTc_4$ as follows:
\begin{enumerate}
\item The hypertree with one hyperedge will be denoted $\Theta^0$.
\item The star tree in $\mathcal{S}_4$ with central vertex $i$ (generally called $\mathrm{S}_4^i$) will be denoted $S^i$. 
\item The line tree in $\mathcal{L}_4$ with hyperedges $\{j,i\},\ \{i,k\},\ \{k,l\}$ will be denoted $L^{i,j}_{k,l}$.
\item The hypertree in $\mathcal{M}_4^1$ which contains the hyperedges $\{i,j\},\ [n]\setminus \{j\}$ will be denoted $\Omega^{i,j}$. 
\end{enumerate}
\end{notation}

\begin{remark}
The following describes the poset structure on the 29 elements of $\HTc_4$ as well as the carrying relation. See also  Figure \ref{fig:htc4}. (Note that each listed partial conjugation might need to replace its domain with its complement to pick the representative not containing the minimal index.) 
\begin{enumerate}
\item $\Theta^0$ is a $\leq$-minimal element that only carries the identity. 
\item $\Omega^{i,j}$ carries only the identity and $x_{i,\{j\}}$. It folds into $\Theta^0$. 
\item $S^i$ carries the Klein 4-group of $\left\{\textrm{id}, x_{i,\{j\}},\ x_{i,\{k\}},\ x_{i,\{j,k\}}\right\}$, where $j$ and $k$ are the non-minimal elements of $[4]\setminus \{i\}$ (and $l$ is the minimal one). It folds into $\Omega^{i,j}$, $\Omega^{i,k}$, $\Omega^{i,l}$, and $\Theta^0$.  
\item $L^{i,j}_{k,l}$ carries the Klein 4-group of  $\left\{\textrm{id}, x_{i,\{j\}},\ x_{k,\{l\}},\ x_{i,\{j\}} x_{k,\{l\}}\right\}$. It folds to $\Omega^{i,j}$, $\Omega^{k,l}$, and $\Theta^0$. 
\end{enumerate}  
\end{remark}

Examining the maximal chains in $\HTc_4$, we see that every simplex in $\HT_4$ has a vertex $\Theta^0$, a vertex of the form $\Omega^{i,j}$, and a vertex of the form either $L^{i,j}_{k,l}$ or $S^{i}$. See Figure \ref{fig:ht4}. Thus, every simplex in $\K_4$ has a vertex $[\alpha, \Theta^0]$, a vertex of the form $[\alpha, \Omega^{i,j}]$, and a vertex of the form either $[\alpha, L^{i,j}_{k,l}]$ or $[\alpha, S^{i}]$ for some $\alpha \in \Out^0(W_4)$. Since the action of $\Out(W_4)$ is transitive on these classes, a fundamental domain for the $\Out(W_4)$ action on $\K_4$ is given by the union of the simplices spanned by $\left\{ \left[\textrm{id},\Theta^0\right],\ \left[\textrm{id},\Omega^{1,3}\right],\ \left[\textrm{id},L^{1,3}_{2,4}\right] \right\}$ (called an L-simplex) and $\left\{ \left[\textrm{id},\Theta^0\right],\ \left[\textrm{id},\Omega^{1,3}\right],\ \left[\textrm{id},S^{1}\right] \right\}$ (called an S-simplex). See Figure \ref{fig:fundom}. 

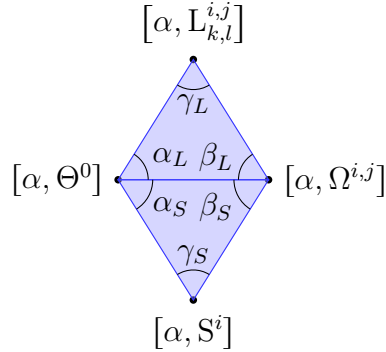
\begin{figure}
\begin{center}
\begin{tikzpicture}[scale=2]

\node (T) [circle,fill,inner sep=1pt,label=180:{$\left[\alpha, \Theta^0\right]$}] at (0, 0) {};
\node (L) [circle,fill,inner sep=1pt,label=90:{$\left[\alpha,\mathrm{L}^{i,j}_{k,l}\right]$}] at (0.5, 0.8) {};
\node (O) [circle,fill,inner sep=1pt,label=0:{$\left[\alpha,\Omega^{i,j}\right]$}] at (1, 0) {};
\node (S) [circle,fill,inner sep=1pt,label=270:{$\left[\alpha,\mathrm{S}^{i}\right]$}] at (0.5, -0.8) {};

\draw [fill,opacity = 0.4, color=blue!40!white] ($(T)$) -- ($(O)$) -- ($(L)$) -- ($(T)$) -- cycle;
\draw [color=blue!80!white] ($(T)$) -- ($(O)$) -- ($(L)$) -- ($(T)$) -- cycle;

\draw [fill,opacity = 0.4, color=blue!40!white] ($(T)$) -- ($(O)$) -- ($(S)$) -- ($(T)$) -- cycle;
\draw [color=blue!80!white] ($(T)$) -- ($(O)$) -- ($(S)$) -- ($(T)$) -- cycle;

\draw ($(T)+(0.2,0)$) arc (0:58:0.2);
\node at ($(T)+(0.35,0.15)$) {$\alpha_L$};
\draw ($(O)+(-0.2,0)$) arc (180:122:0.2);
\node at ($(O)+(-0.35,0.15)$) {$\beta_L$};
\draw ($(L)+(-0.106,-0.1696)$) arc (-122:-58:0.2);
\node at ($(L)+(0,-0.3)$) {$\gamma_L$};
\draw ($(T)+(0.23,0)$) arc (0:-58:0.23);
\node at ($(T)+(0.35,-0.18)$) {$\alpha_S$};
\draw ($(O)+(-0.23,0)$) arc (180:238:0.23);
\node at ($(O)+(-0.35,-0.18)$) {$\beta_S$};
\draw ($(S)+(-0.106,0.1696)$) arc (122:58:0.2);
\node at ($(S)+(0,0.3)$) {$\gamma_S$};

\end{tikzpicture}
\end{center}
\caption[The Fundamental Domain for the $\Out(W_4)$ Action on $\K_4$]{The fundamental domain for the $\Out(W_4)$ action on $\K_4$, with associated angles after metrizing.}
\label{fig:fundom}
\end{figure}

This description of $\K_4$ and the action of $\Out(W_n)$ will be useful in Section \ref{sec:metmm4}.

\section{Algebra of Out($W_n$)}
A presentation for $\Aut^0(W_\Gamma)$ (what M{\"u}hlherr calls $\Spe(W)$) is given in \cite{Muhlherr} as a semidirect product $\Inn(W_\Gamma) \rtimes \Out^0(W_\Gamma)$ and so a finite presentation can be extracted for $\Out^0(W_\Gamma)$ (after a few elementary Tietze transformations). 

Recall that a generating set for $\Out^0(W_n)$ is given by the set of partial conjugations $\mathcal{P}^0 = \{x_{i,D}\}$, where $i \in [n]$, $j$ is the minimal index in $[n] \setminus \{i\}$,  and $D$ is a non-empty subset of $[n] \setminus \{i,j\}$ (\cite{GPR}) . Also remember that $\widetilde{D} = D \cup \{i\}$ and $\widetilde{D^c} = D^c \cup \{i\}$. There are also some obvious classes of relations in $\Out^0(W_n)$: 

\begin{enumerate}
\item (R1) $x_{i,D} x_{i,D} = \textrm{id}$
\item (R2) $x_{i,D} x_{i,D'} = x_{i,(D \cup D') \setminus (D \cap D')}$
\item (R3) $[x_{i,D_i}, x_{j,D_j}] = 1$ if $\widetilde{D_i} \cap \widetilde{D_j} = \varnothing$, $\widetilde{D_i^c} \cap \widetilde{D_j} = \varnothing$, or $\widetilde{D_i} \cap \widetilde{D_j^c} = \varnothing$. (See Lemma \ref{lem:compc}). 
\end{enumerate}

Some elementary Tietze transformations on the main Theorem in \cite{Muhlherr} give the following. 

\begin{theorem}[M{\"u}hlherr \cite{Muhlherr}]
\label{thm:pres}
 A finite presentation for $\Out^0(W_n)$ is given by the generators $\mathcal{P}^0$ and the set of relations given by the union of the classes (R1), (R2), and (R3). 
\end{theorem}

\begin{remark}
\label{rem:presout3}
Following the presentation of $\Out^0(W_3)$ from Theorem \ref{thm:pres} (and using generators $y_{i,D}$ instead of $x_{i,D}$ to distinguish the $n=3$ and $n=4$ cases), we find that $\mathcal{P}^0 = \left\{y_{1,\{3\}}, y_{2,\{3\}}, y_{3,\{2\}}\right\}$. For each of these automorphisms, $\widetilde{D}$ and $\widetilde{D^c}$ contain at least two indices each, but since there are only three indices total in $[3]$, these extended domains can never be disjoint. Thus, there are no relations of the form (R3). Also, there are no partial conjugations in $\mathcal{P}^0$ that have the same acting letter, and so there are no relations of the form (R2) either. Thus, the full presentation for $\Out^0(W_3)$ is given by:
\[
\Out^0(W_3) = \left\langle y_{1,\{3\}}, y_{2,\{3\}}, y_{3,\{2\}} \mid y_{1,\{3\}}^2 = y_{2,\{3\}}^2 = y_{3,\{2\}}^2 = \textrm{id} \right\rangle \cong W_3.
\]
Thus, $\Out^0(W_3) \cong W_3$ and so is a right-angled Coxeter group. However, the higher rank $\Out^0(W_n)$ are \emph{not} themselves right-angled Coxeter groups \cite{mythesis}. 
\end{remark}

\section{Metrizing McCullough-Miller 4-Space}\label{sec:metmm4}

In this section, we show that $\K_n$ admits no $G$-equivariant $\cat(\kappa)$ $M_{\kappa}$-polyhedral structure for $G \cong \Out(W_n)$ or $G \cong \Out^0(W_n)$, $n \geq 4$, and $\kappa \leq 0$.

This is analogous to a result in Bridson's thesis \cite{Bridson} for $\Out(F_n)$ (for $n \geq 3$). 

We shall need the following foundational theorem on curvature in polyhedral complexes. Gromov stated it without proof in \cite{Gromov}, and Bridson proved it in full generality in \cite{Bridson}. 

\begin{theorem}[Gromov's Link Condition \cite{Gromov,Bridson,BH}]\label{thm:link}
For $\kappa \leq 0$, a 2-dimensional $M_{\kappa}$-complex with finitely many isometry classes of polyhedrons is $\cat(\kappa)$ if and only if it is simply connected and the link of each vertex is globally $\cat(1)$ if and only if it is simply connected and for each vertex $v$, every injective loop in the link of $v$, $\Lk(v)$, has length at least $2 \pi$. 
\end{theorem}

For 2-dimensional complexes, this condition reduces to the following. 

\begin{theorem}[Gromov, Bridson \cite{Gromov, Bridson,BH}]\label{thm:brid}
If $X$ is a 2-dimensional $\cat(\kappa)$ simplicial $M_{\kappa}$-complex for $\kappa \leq 0$, and $\alpha_i \in (0,\pi]$ are the angles at each corner of a simplex in the complex, then $\Sigma_T \alpha_i \leq \pi$, where $T$ are the interior angles of a simplex, and $\Sigma_{\gamma} \alpha_i \geq 2 \pi$, where $\gamma$ are the angles around an injective loop in a link of a vertex. 
\end{theorem}

In particular, if the system of inequalities in the $\alpha_i$ given by Theorem \ref{thm:brid} is unsatisfiable, then $X$ admits no $M_{\kappa}$-polyhedral structure of non-positive curvature.

\subsection{The Out($W_4$) Case}
\label{subsec:out}

We would now like to use Theorem \ref{thm:brid} to show that no appropriate $\cat(\kappa)$ metric can be assigned to $\K_4$. So suppose that $\K_4$ \emph{has} been given an $\Out(W_4)$-equivariant metric that makes $\K_4$ a $\cat(\kappa)$ $M_{\kappa}$-simplicial complex. 

\begin{definition}
\label{def:ang}
Since the metric is $\Out(W_4)$-equivariant, it suffices to assign an angle to each corner of each simplex in the fundamental domain of the action in order to specify an angle in every corner of every simplex of $\K_4$. So let the angles be defined as follows:
\begin{enumerate}
\item In any L-simplex, let $\alpha_L$ be the vertex angle of $\left[\alpha,\Theta^0 \right]$, let $\beta_L$ be the vertex angle of $\left[\alpha,\Omega^{i,j}\right]$, and let $\gamma_L$ be the vertex angle of $\left[\alpha,L^{i,j}_{k,l}\right]$.
\item In any S-simplex, let $\alpha_S$ be the vertex angle of $\left[\alpha,\Theta^0\right]$, let $\beta_S$ be the vertex angle of $\left[\alpha,\Omega^{i,j}\right]$, and let $\gamma_S$ be the vertex angle of $\left[\alpha,S^{i}\right]$.
\end{enumerate}
\end{definition}

By Theorem \ref{thm:brid}, we know that the angles must satisfy the following inequalities:
\begin{align}\label{eq:TL}
\alpha_L + \beta_L + \gamma_L \leq \pi
\end{align}
\begin{align}\label{eq:TS}
\alpha_S + \beta_S + \gamma_S \leq \pi
\end{align}

To determine the other inequalities, we need to understand what the links of the vertices in $\K_4$ look like. It suffices to consider the links of the vertices in a fundamental domain. 

\begin{example}
We start with the link of $\left[\textrm{id},\Omega^{1,3}\right]$.

In $\K_4$, $\left[\textrm{id},\Omega^{1,3}\right]$ is adjacent to $\left[\alpha, \Lambda \right]$ whenever either $\Omega^{1,3} \leq \Lambda$ and $\textrm{id}^{-1} \alpha = \alpha$ is carried by $\Lambda$, or else $\Theta \leq \Omega^{1,3}$ and $\textrm{id}^{-1} \alpha = \alpha$ is carried by $\Omega^{1,3}$. In the former case, since $\left[\textrm{id},\Lambda\right] = \left[\alpha,\Lambda\right]$ for any $\alpha$ carried by $\Lambda$, it suffices to consider the representatives $\left[\textrm{id},\Lambda\right]$. In the latter case, $\alpha$ might not be carried by $\Theta$, so the different $\left[\alpha, \Theta\right]$ will result in different vertices. 

Since $\Omega^{1,3}$ carries the identity and $x_{1,\{3\}}$, $\left[\textrm{id},\Omega^{1,3}\right] = \left[x_{1,\{3\}},\Omega^{1,3}\right]$, and so $\left[\textrm{id},\Omega^{1,3}\right]$ is adjacent to $\left[\textrm{id},\Theta^0\right]$ and $\left[x_{1,\{3\}},\Theta^0\right]$, which are different vertices. 

On the other hand, the hypertrees greater than $\Omega^{1,3}$ in $\HTc_4$ are the ones that it can unfold into, namely, $L^{1,3}_{2,4}$, $L^{1,3}_{4,2}$, and $S^{1}$. So $\left[\textrm{id},\Omega^{1,3}\right]$ is also adjacent to $\left[\textrm{id},L^{1,3}_{2,4}\right]$, $\left[\textrm{id},L^{1,3}_{4,2}\right]$, and $\left[\textrm{id},S^{1}\right]$. These 5 vertices are the only ones adjacent to $\Omega^{1,3}$ in $\K_4$. 

In the link, vertices are connected by an edge if they share a simplex in $\K_4$ and the length of that edge is given by the angle with vertex $\left[\textrm{id},\Omega^{1,3}\right]$ in that simplex. So the line trees and star tree are never connected to each other, but the nuclear vertex is connected to each whenever the label matches up. The link is shown in Figure \ref{fig:lkomeg}. 

Reading off the injective loops that go around the large square as well as one of the smaller squares, we use Theorem \ref{thm:brid} to get the inequalities:

\begin{align}
\label{eq:bl}
\beta_L + \beta_L + \beta_L + \beta_L &\geq 2 \pi \\ 
\textrm{i.e., }\beta_L &\geq \frac{\pi}{2} \nonumber
\end{align}

\begin{align}
\label{eq:blbs}
\beta_L + \beta_L + \beta_S + \beta_S &\geq 2 \pi
\\ 
\textrm{i.e., }\beta_L + \beta_S &\geq \pi \nonumber
\end{align}

\end{example}

\begin{figure}

\hspace*{-2cm}
\begin{tikzpicture}[scale=1.5]
\centering

\node (S1) [star,fill,inner sep=2pt,label=0:{$\left[\mathrm{id},\mathrm{S}^{1}\right]$}] at (0, 0) {};
\node (S1L) at ($(S1)+(0.75,-0.4)$) {$= \left[x_{1,\{3\}},\mathrm{S}^{1}\right]$};
\node (1324) [circle,fill,inner sep=1.5pt,label=0:{$\left[\mathrm{id},\mathrm{L}^{1,3}_{2,4}\right] = \left[x_{1,\{3\}},\mathrm{L}^{1,2}_{3,4}\right]$}] at (2, 0) {};
\node (1342) [circle,fill,inner sep=1.5pt,label=180:{$\left[\mathrm{id},\mathrm{L}^{1,3}_{4,2}\right] = \left[x_{1,\{3\}},\mathrm{L}^{1,3}_{4,2}\right]$}] at (-2, 0) {};
\node (T1) [circle,fill,inner sep=1.5pt,label=-90:{$\left[\mathrm{id},\Theta^0\right]$}] at (0, -2) {};
\node (T2) [circle,fill,inner sep=1.5pt,label=90:{$\left[x_{1,\{3\}},\Theta^0\right]$}] at (0, 2) {};

\draw [color=black] ($(T1)$) -- ($(S1)$);
\draw [color=black] ($(T1)$) -- ($(1324)$);
\draw [color=black] ($(T1)$) -- ($(1342)$);
\draw [color=black] ($(T2)$) -- ($(S1)$);
\draw [color=black] ($(T2)$) -- ($(1324)$);
\draw [color=black] ($(T2)$) -- ($(1342)$);

\node [diamond,fill,color=ForestGreen,inner sep=2.4pt] at ($(1324)$) {};
\node [diamond,fill,color=ForestGreen,inner sep=2.4pt] at ($(1342)$) {};

\node [circle,fill,color=Plum,inner sep=2.6pt] at ($(T1)$) {};
\node [circle,fill,color=Plum,inner sep=2.6pt] at ($(T2)$) {};

\node [star,fill,color=blue!80!white,inner sep=3pt] at ($(S1)$) {};

\end{tikzpicture}

\caption[The Link of $\Omega^{1,3}$]{The link of $\left[\textrm{id},\Omega^{1,3}\right] = \left[x_{1,\{3\}},\Omega^{1,3}\right]$ in $\K_4$. The blue stars are star trees, the green diamonds are line trees, and the purple circles are nuclear vertices.}
\label{fig:lkomeg}
\end{figure}
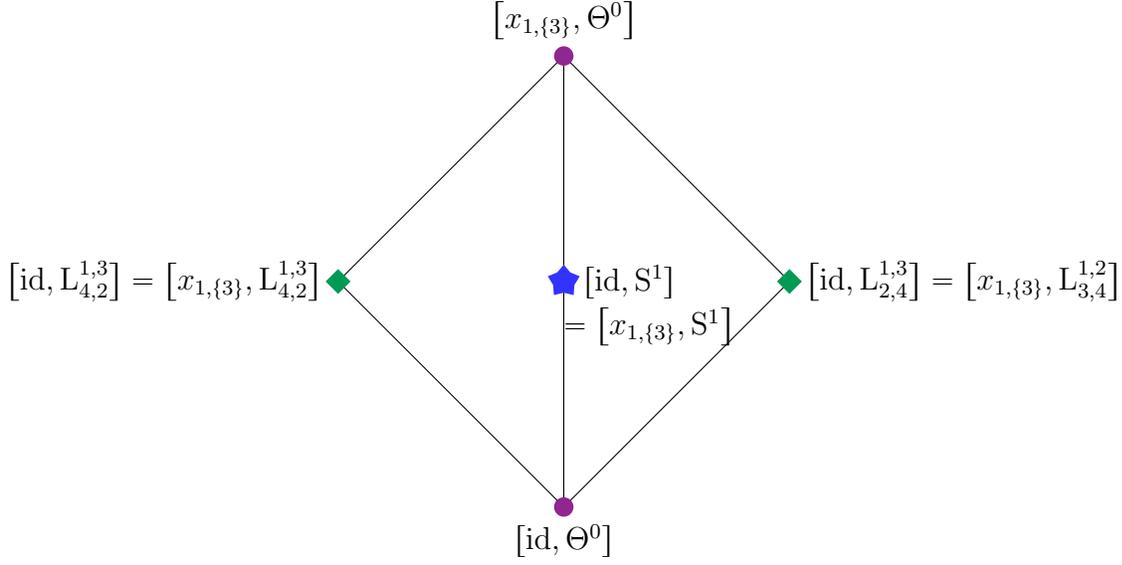

\begin{example}
Next, we examine the link of $\left[\textrm{id},L^{1,3}_{2,4}\right]$.

Since $L^{1,3}_{2,4}$ carries $\left\{\textrm{id},\  x_{1,\{3\}},\  x_{2,\{4\}},\  x_{1,\{3\}} x_{2,\{4\}}\right\}$, $\left[\textrm{id},L^{1,3}_{2,4}\right] = \left[x_{1,\{3\}},L^{1,3}_{2,4}\right] = \left[x_{2,\{4\}},L^{1,3}_{2,4}\right] = \left[x_{1,\{3\}} x_{2,\{4\}},L^{1,3}_{2,4}\right]$, and so $\left[\textrm{id},L^{1,3}_{2,4}\right]$ is adjacent to $\left[\textrm{id},\Theta^0\right]$, $\left[x_{1,\{3\}},\Theta^0\right]$, $\left[x_{2,\{4\}},\Theta^0\right]$, and $\left[x_{1,\{3\}}x_{2,\{4\}},\Theta^0\right]$ which are different vertices. 

$\left[\textrm{id},L^{1,3}_{2,4}\right]$ is also adjacent to the vertices with these same four labels and with hypertree $\Omega^{1,3}$ or $\Omega^{2,4}$, but since each of these vertices has two representatives (e.g., $\left[\textrm{id},\Omega^{1,3}\right] = \left[x_{1,\{3\}},\Omega^{1,3}\right]$), this results in only four new adjacent vertices in $\K_4$. 

In total, there are 8 adjacent vertices. In the L-simplices in $\K_4$, the nuclear vertices are connected to both $\Omega^{i,j}$ vertices, and each of those vertices carry one non-identity automorphism, and so are connected to two nuclear vertices. Calculating all of these adjacencies, we see that the link graph is a single cycle of length 8, as shown in Figure \ref{fig:lkline}. 

Reading off the single injective loops in the cycle, we use Theorem \ref{thm:brid} to get the inequality:

\begin{align}
\label{eq:gl}
8 \gamma_L &\geq 2 \pi \\ 
\textrm{i.e., }\gamma_L &\geq \frac{\pi}{4} \nonumber
\end{align}

\end{example}

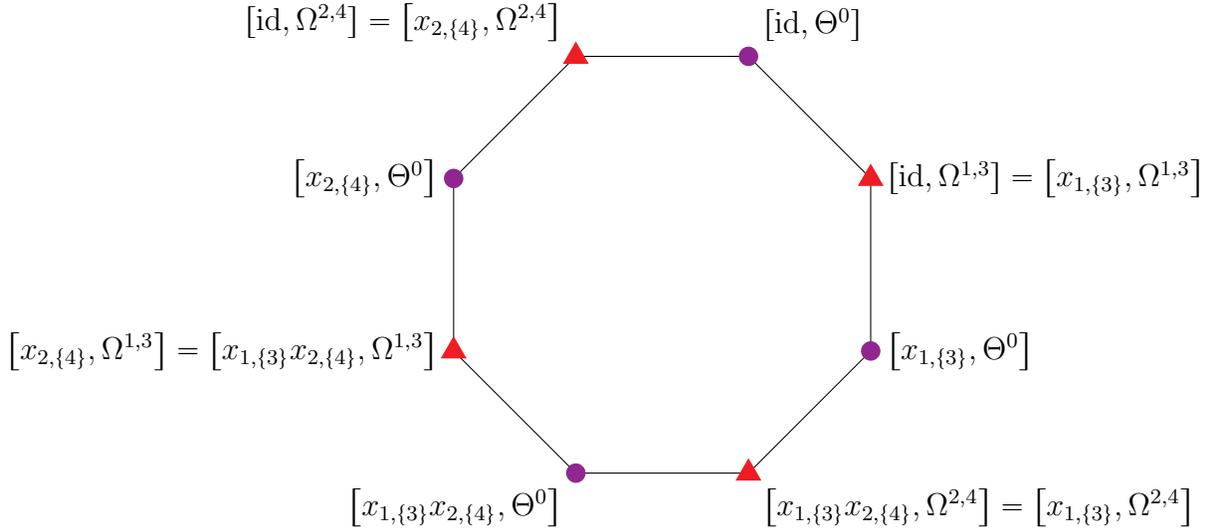
\begin{figure}
\hspace*{-2cm}
\begin{tikzpicture}[scale=1.5]
\centering

\node (13I) [circle,fill,inner sep=1.5pt,label=0:{$\left[\mathrm{id},\Omega^{1,3}\right] = \left[x_{1,\{3\}},\Omega^{1,3}\right]$}] at (1.848, 0.765) {};
\node (TI) [circle,fill,inner sep=1.5pt,label=45:{$\left[\mathrm{id},\Theta^0 \right]$}] at (0.765, 1.848) {};
\node (24I) [circle,fill,inner sep=1.5pt,label=135:{$\left[\mathrm{id},\Omega^{2,4}\right] = \left[x_{2,\{4\}},\Omega^{2,4}\right]$}] at (-0.765, 1.848) {};
\node (T24) [circle,fill,inner sep=1.5pt,label=180:{$\left[x_{2,\{4\}},\Theta^0\right]$}] at (-1.848, 0.765) {};
\node (1324) [circle,fill,inner sep=1.5pt,label=180:{$\left[x_{2,\{4\}},\Omega^{1,3}\right] = \left[x_{1,\{3\}}x_{2,\{4\}},\Omega^{1,3}\right]$}] at (-1.848, -0.765) {};
\node (T1324) [circle,fill,inner sep=1.5pt,label=225:{$\left[x_{1,\{3\}}x_{2,\{4\}},\Theta^0\right]$}] at (-0.765, -1.848) {};
\node (2413) [circle,fill,inner sep=1.5pt,label=-45:{$\left[x_{1,\{3\}}x_{2,\{4\}},\Omega^{2,4}\right] = \left[x_{1,\{3\}},\Omega^{2,4}\right]$}] at (0.765, -1.848) {};
\node (T13) [circle,fill,inner sep=1.5pt,label=0:{$\left[x_{1,\{3\}},\Theta^0\right]$}] at (1.848, -0.765) {};

\draw [color=black] ($(13I)$) -- ($(TI)$);
\draw [color=black] ($(24I)$) -- ($(TI)$);
\draw [color=black] ($(24I)$) -- ($(T24)$);
\draw [color=black] ($(1324)$) -- ($(T24)$);
\draw [color=black] ($(1324)$) -- ($(T1324)$);
\draw [color=black] ($(2413)$) -- ($(T1324)$);
\draw [color=black] ($(2413)$) -- ($(T13)$);
\draw [color=black] ($(13I)$) -- ($(T13)$);

\node [regular polygon,regular polygon sides=3,fill,color=Red,inner sep=2pt] at ($(13I)$) {};
\node [regular polygon,regular polygon sides=3,fill,color=Red,inner sep=2pt] at ($(24I)$) {};
\node [regular polygon,regular polygon sides=3,fill,color=Red,inner sep=2pt] at ($(1324)$) {};
\node [regular polygon,regular polygon sides=3,fill,color=Red,inner sep=2pt] at ($(2413)$) {};

\node [circle,fill,color=Plum,inner sep=2.6pt] at ($(TI)$) {};
\node [circle,fill,color=Plum,inner sep=2.6pt] at ($(T24)$) {};
\node [circle,fill,color=Plum,inner sep=2.6pt] at ($(T1324)$) {};
\node [circle,fill,color=Plum,inner sep=2.6pt] at ($(T13)$) {};

\end{tikzpicture}

\caption[The Link of $\mathrm{L}^{1,3}_{2,4}$]{The link of $\left[\textrm{id},\mathrm{L}^{1,3}_{2,4}\right] = \left[x_{1,\{3\}},\mathrm{L}^{1,3}_{2,4}\right]  = \left[x_{2,\{4\}},\mathrm{L}^{1,3}_{2,4}\right] = \left[x_{1,\{3\}}x_{2,\{4\}},\mathrm{L}^{1,3}_{2,4}\right]$ in $\K_4$. The purple circles are nuclear vertices, and the red triangles are elements of $\mathcal{M}_4^1$.}
\label{fig:lkline}
\end{figure}

\begin{example}
Now, we construct the link of $\left[\textrm{id},S^{1}\right]$.

Since $S^{1}$ carries $\left\{\textrm{id},\  x_{1,\{3\}},\  x_{1,\{4\}},\  x_{1,\{3\}} x_{1,\{4\}}\right\}$, $\left[\textrm{id},S^{1}\right] = \left[x_{1,\{3\}},S^{1}\right] = \left[x_{1,\{4\}},S^{1}\right] = \left[x_{1,\{3\}} x_{1,\{4\}},S^{1}\right]$, and so $\left[\textrm{id},S^{1}\right]$ is adjacent to $\left[\textrm{id},\Theta^0\right]$, $\left[x_{1,\{3\}},\Theta^0\right]$, $\left[x_{1,\{4\}},\Theta^0\right]$, and $\left[x_{1,\{3\}}x_{1,\{4\}},\Theta^0\right]$ which are different vertices. 

$\left[\textrm{id},S^{1}\right]$ is also adjacent to the vertices with these same four labels and with hypertree $\Omega^{1,2}$, $\Omega^{1,3}$, or  $\Omega^{1,4}$, but since each of these vertices has two representatives (e.g., $\left[\textrm{id},\Omega^{1,3}\right] = \left[x_{1,\{3\}},\Omega^{1,3}\right]$), this results in only six new adjacent vertices in $\K_4$. 

In total, there are 10 adjacent vertices. In the S-simplices in $\K_4$, the nuclear vertices are connected to all three $\Omega^{i,j}$ vertices, and each of those vertices carry one non-identity automorphism, and so are connected to two nuclear vertices. Calculating all of these adjacencies, we see that the link graph is three cycles of length 6, each glued to each other along paths of length 2, as shown in Figure \ref{fig:lkstar}. 

Since all of the edges in the link have length $\gamma_S$, finding the smallest injective loop will give us an inequality that will imply all of the others. So, reading off the smallest injective loop in the link, which is a cycle of length 6, we use Theorem \ref{thm:brid} to get the inequality:

\begin{align}
\label{eq:gs}
6 \gamma_S &\geq 2 \pi \\ 
\textrm{i.e., }\gamma_S &\geq \frac{\pi}{3} \nonumber
\end{align}

\end{example}

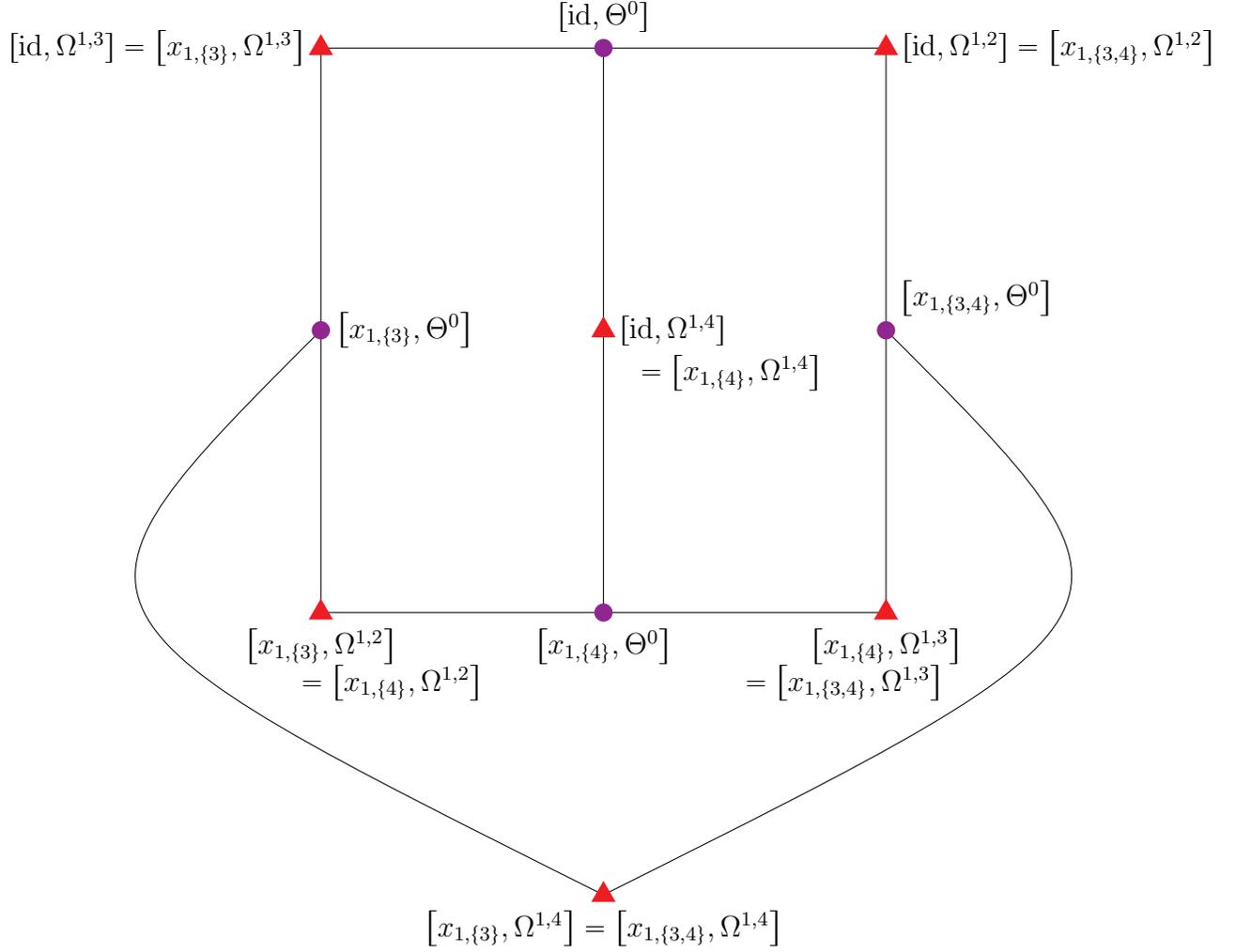
\begin{figure}
\hspace*{-2cm}
\begin{tikzpicture}[scale=2]
\centering

\node (14I) [circle,fill,inner sep=1.5pt,label=0:{$\left[\mathrm{id},\Omega^{1,4}\right]$}] at (0,0) {};

\node (14IL) at ($(14I)+(0.9,-0.3)$) {$= \left[x_{1,\{4\}},\Omega^{1,4}\right]$};

\node (TI) [circle,fill,inner sep=1.5pt,label=90:{$\left[\mathrm{id},\Theta^0 \right]$}] at (0, 2) {};
\node (13I) [circle,fill,inner sep=1.5pt,label=180:{$\left[\mathrm{id},\Omega^{1,3}\right] = \left[x_{1,\{3\}},\Omega^{1,3}\right]$}] at (-2,2) {};
\node (T13) [circle,fill,inner sep=1.5pt,label=0:{$\left[x_{1,\{3\}},\Theta^0 \right]$}] at (-2, 0) {};
\node (1213) [circle,fill,inner sep=1.5pt,label=-90:{$\left[x_{1,\{3\}},\Omega^{1,2}\right]$}] at (-2,-2) {};

\node (1213L) at ($(1213)+(0.5,-0.5)$) {$= \left[x_{1,\{4\}},\Omega^{1,2}\right]$};

\node (T14) [circle,fill,inner sep=1.5pt,label=-90:{$\left[x_{1,\{4\}},\Theta^0 \right]$}] at (0, -2) {};
\node (1314) [circle,fill,inner sep=1.5pt,label=-90:{$\left[x_{1,\{4\}},\Omega^{1,3}\right]$}] at (2,-2) {};

\node (1314L) at ($(1314)+(-0.3,-0.5)$) {$= \left[x_{1,\{3,4\}},\Omega^{1,3}\right]$};

\node (T134) [circle,fill,inner sep=1.5pt,label=45:{$\left[x_{1,\{3,4\}},\Theta^0 \right]$}] at (2, 0) {};
\node (12I) [circle,fill,inner sep=1.5pt,label=0:{$\left[\textrm{id},\Omega^{1,2}\right] = \left[x_{1,\{3,4\}},\Omega^{1,2}\right]$}] at (2,2) {};
\node (1413) [circle,fill,inner sep=1.5pt,label=-90:{$\left[x_{1,\{3\}},\Omega^{1,4}\right] = \left[x_{1,\{3,4\}},\Omega^{1,4}\right]$}] at (0,-4) {};

\draw [color=black] ($(14I)$) -- ($(TI)$);
\draw [color=black] ($(13I)$) -- ($(TI)$);
\draw [color=black] ($(13I)$) -- ($(T13)$);
\draw [color=black] ($(1213)$) -- ($(T13)$);
\draw [color=black] ($(1213)$) -- ($(T14)$);
\draw [color=black] ($(1314)$) -- ($(T14)$);
\draw [color=black] ($(1314)$) -- ($(T134)$);
\draw [color=black] ($(12I)$) -- ($(T134)$);
\draw [color=black] ($(12I)$) -- ($(TI)$);
\draw [color=black] ($(14I)$) -- ($(T14)$);

\draw [color=black] ($(T13)$) .. controls (-4,-2) .. ($(1413)$);
\draw [color=black] ($(T134)$) .. controls (4,-2) .. ($(1413)$);

\node [regular polygon,regular polygon sides=3,fill,color=Red,inner sep=2pt] at ($(14I)$) {};
\node [regular polygon,regular polygon sides=3,fill,color=Red,inner sep=2pt] at ($(13I)$) {};
\node [regular polygon,regular polygon sides=3,fill,color=Red,inner sep=2pt] at ($(1213)$) {};
\node [regular polygon,regular polygon sides=3,fill,color=Red,inner sep=2pt] at ($(1314)$) {};
\node [regular polygon,regular polygon sides=3,fill,color=Red,inner sep=2pt] at ($(12I)$) {};
\node [regular polygon,regular polygon sides=3,fill,color=Red,inner sep=2pt] at ($(1413)$) {};

\node [circle,fill,color=Plum,inner sep=2.6pt] at ($(TI)$) {};
\node [circle,fill,color=Plum,inner sep=2.6pt] at ($(T13)$) {};
\node [circle,fill,color=Plum,inner sep=2.6pt] at ($(T134)$) {};
\node [circle,fill,color=Plum,inner sep=2.6pt] at ($(T14)$) {};

\end{tikzpicture}

\caption[The Link of $\textrm{S}^{1}$]{The link of $\left[\textrm{id},\textrm{S}^{1}\right] = \left[x_{1,\{3\}},\textrm{S}^{1}\right]  = \left[x_{1,\{4\}},\textrm{S}^{1}\right] = \left[x_{1,\{3,4\}},\textrm{S}^{1}\right]$ in $\K_4$. It consists of three hexagons glued together. The purple circles are nuclear vertices, and the red triangles are elements of $\mathcal{M}_4^1$.}
\label{fig:lkstar}
\end{figure}

\begin{example}
Finally, we construct the link of $\left[\textrm{id},\Theta^0\right]$.

Since $\Theta^0$ only carries the identity but is in every simplex in $\HT_4$, $\left[\textrm{id},\Theta^0\right]$ is adjacent only to vertices with the same label but \emph{any} hypertree, i.e., the vertices $\left\{\left[\textrm{id},\Lambda\right]\mid \Lambda \in \HTc_4 \right\}$ in $\K_4$. So its link in $\K_4$ is identical to its link in $\HT_4$, which is given in Figure 5 in \cite{Piggott} and reproduced below in Figure \ref{fig:lknuc}. 

It has 4 star vertices, 12 omega vertices, and 12 line vertices, for a total of 28. The star vertices are each connected to three omega vertices, the line vertices are each connected to two omega vertices, and the omega vertices are each connected to one star and two line vertices. The link is made up of glued octagons, and we only need the smallest injective loops which wrap around each octagon. There are two types, so we once again use Theorem \ref{thm:brid} to get the inequalities:

\begin{align}
\label{eq:al}
8 \alpha_L &\geq 2 \pi \\ 
\textrm{i.e., }\alpha_L &\geq \frac{\pi}{4} \nonumber
\end{align}

\begin{align}
\label{eq:alas}
4 \alpha_L + 4 \alpha_S &\geq 2 \pi \\ 
\textrm{i.e., }\alpha_L + \alpha_S &\geq \frac{\pi}{2} \nonumber
\end{align}

\end{example}

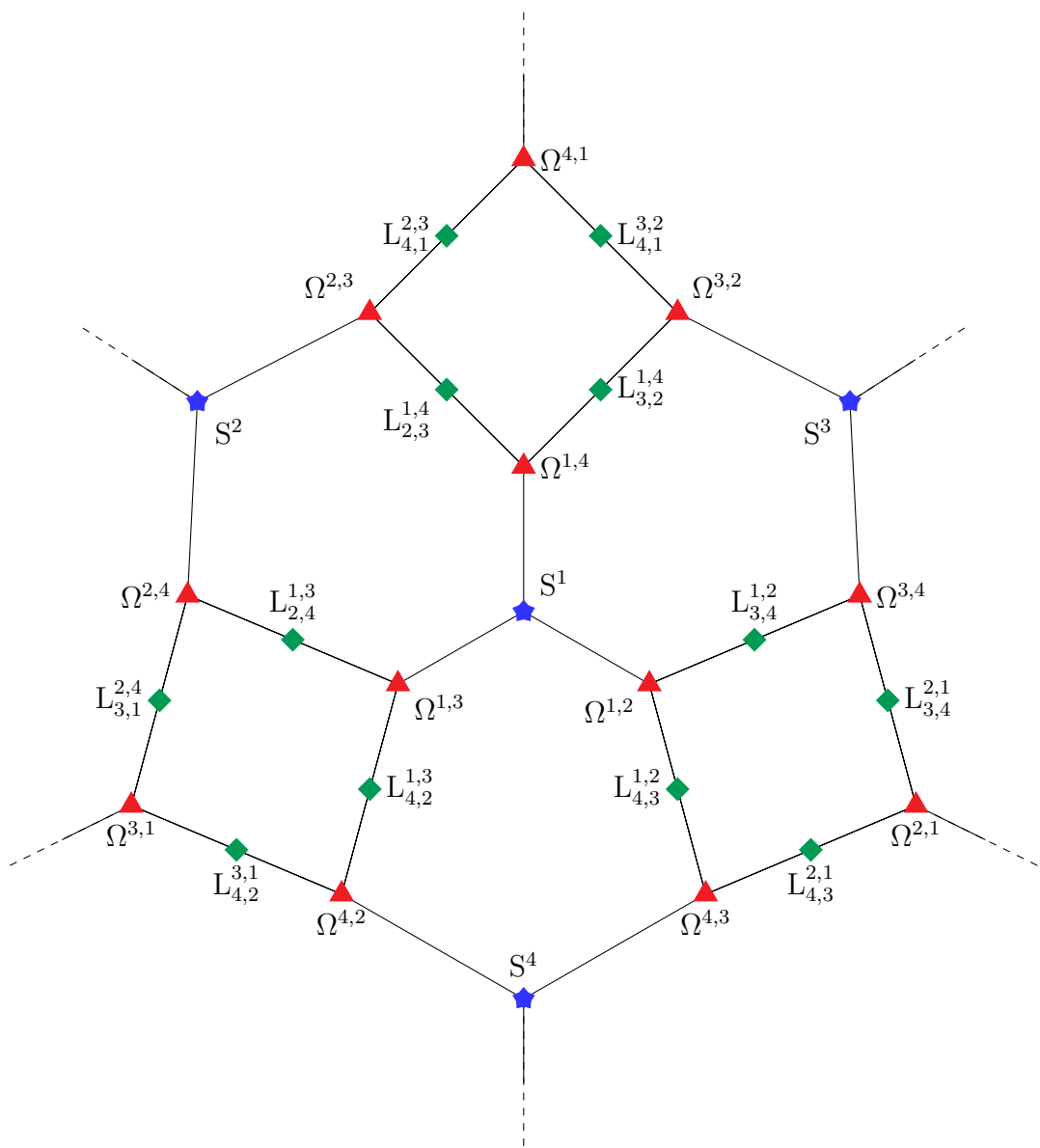
\begin{figure}
\hspace*{-3.5cm}%
\begin{tikzpicture}[scale=2]
\centering

\node (14) [circle,fill,inner sep=1.5pt,label=0:{$\Omega^{1,4}$}] at (0, 1) {};
\node (13) [circle,fill,inner sep=1.5pt,label=-5:{$\Omega^{1,3}$}] at (-0.866, -0.5) {};
\node (134) [circle,fill,inner sep=1.5pt,label=225:{$\Omega^{1,2}$}] at (0.866, -0.5) {};
\node (S1) [star,fill,inner sep=2pt,label=45:{$\mathrm{S}^{1}$}] at (0, 0) {};
\draw [color=black] ($(14)$) -- ($(S1)$);
\draw [color=black] ($(13)$) -- ($(S1)$);
\draw [color=black] ($(134)$) -- ($(S1)$);

\node (32) [circle,fill,inner sep=1.5pt,label=45:{$\Omega^{3,2}$}] at (1.0607, 2.0607) {};
\node (1432) [circle,fill,inner sep=1.5pt,label=0:{$\mathrm{L}^{1,4}_{3,2}$}] at ($0.5*(14) + 0.5*(32)$) {};
\draw [color=black] ($(14)$) -- ($(32)$) -- ($(1432)$) -- ($(14)$) -- cycle;

\node (23) [circle,fill,inner sep=1.5pt,label=135:{$\Omega^{2,3}$}] at (-1.0607, 2.0607) {};
\node (1423) [circle,fill,inner sep=1.5pt,label=200:{$\mathrm{L}^{1,4}_{2,3}$}] at ($0.5*(14) + 0.5*(23)$) {};
\draw [color=black] ($(14)$) -- ($(23)$) -- ($(1423)$) -- ($(14)$) -- cycle;

\node (423) [circle,fill,inner sep=1.5pt,label=0:{$\Omega^{4,1}$}] at (0, 3.1214) {};
\node (23423) [circle,fill,inner sep=1.5pt,label=180:{$\mathrm{L}^{2,3}_{4,1}$}] at ($0.5*(23) + 0.5*(423)$) {};
\node (32423) [circle,fill,inner sep=1.5pt,label=0:{$\mathrm{L}^{3,2}_{4,1}$}] at ($0.5*(32) + 0.5*(423)$) {};
\draw [color=black] ($(423)$) -- ($(23)$) -- ($(23423)$) -- ($(423)$) -- cycle;
\draw [color=black] ($(423)$) -- ($(32)$) -- ($(32423)$) -- ($(423)$) -- cycle;

\node (24) [circle,fill,inner sep=1.5pt,label=180:{$\Omega^{2,4}$}] at (-2.315, 0.1118) {};
\node (1324) [circle,fill,inner sep=1.5pt,label=90:{$\mathrm{L}^{1,3}_{2,4}$}] at ($0.5*(13) + 0.5*(24)$) {};
\draw [color=black] ($(13)$) -- ($(24)$) -- ($(1324)$) -- ($(13)$) -- cycle;

\node (42) [circle,fill,inner sep=1.5pt,label=270:{$\Omega^{4,2}$}] at (-1.2542, -1.94889) {};
\node (1342) [circle,fill,inner sep=1.5pt,label=0:{$\mathrm{L}^{1,3}_{4,2}$}] at ($0.5*(13) + 0.5*(42)$) {};
\draw [color=black] ($(13)$) -- ($(42)$) -- ($(1342)$) -- ($(13)$) -- cycle;

\node (324) [circle,fill,inner sep=1.5pt,label=270:{$\Omega^{3,1}$}] at (-2.7032, -1.3371) {};
\node (24324) [circle,fill,inner sep=1.5pt,label=180:{$\mathrm{L}^{2,4}_{3,1}$}] at ($0.5*(24) + 0.5*(324)$) {};
\node (32442) [circle,fill,inner sep=1.5pt,label=270:{$\mathrm{L}^{3,1}_{4,2}$}] at ($0.5*(324) + 0.5*(42)$) {};
\draw [color=black] ($(324)$) -- ($(24)$) -- ($(24324)$) -- ($(324)$) -- cycle;
\draw [color=black] ($(324)$) -- ($(42)$) -- ($(32442)$) -- ($(324)$) -- cycle;

\node (34) [circle,fill,inner sep=1.5pt,label=0:{$\Omega^{3,4}$}] at (2.315, 0.1118) {};
\node (13434) [circle,fill,inner sep=1.5pt,label=90:{$\mathrm{L}^{1,2}_{3,4}$}] at ($0.5*(134) + 0.5*(34)$) {};
\draw [color=black] ($(134)$) -- ($(34)$) -- ($(13434)$) -- ($(134)$) -- cycle;

\node (43) [circle,fill,inner sep=1.5pt,label=270:{$\Omega^{4,3}$}] at (1.2542, -1.94889) {};
\node (13443) [circle,fill,inner sep=1.5pt,label=180:{$\mathrm{L}^{1,2}_{4,3}$}] at ($0.5*(134) + 0.5*(43)$) {};
\draw [color=black] ($(134)$) -- ($(43)$) -- ($(13443)$) -- ($(134)$) -- cycle;

\node (234) [circle,fill,inner sep=1.5pt,label=270:{$\Omega^{2,1}$}] at (2.7032, -1.3371) {};
\node (23434) [circle,fill,inner sep=1.5pt,label=0:{$\mathrm{L}^{2,1}_{3,4}$}] at ($0.5*(234) + 0.5*(34)$) {};
\node (23443) [circle,fill,inner sep=1.5pt,label=270:{$\mathrm{L}^{2,1}_{4,3}$}] at ($0.5*(234) + 0.5*(43)$) {};
\draw [color=black] ($(234)$) -- ($(34)$) -- ($(23434)$) -- ($(234)$) -- cycle;
\draw [color=black] ($(234)$) -- ($(43)$) -- ($(23443)$) -- ($(234)$) -- cycle;

\node (V234) at (-3.3757,2.173) {};
\node (V23) at (5.0161,-1.18829) {};
\node (V24) at (3.9885,-3.2657) {};

\node (S2) [star,fill,inner sep=2pt,label=-45:{$\mathrm{S}^{2}$}] at ($0.333*(V234) + 0.333*(23) + 0.333*(24)$) {};
\node (VS2) at ($0.333*(234) + 0.333*(V23) + 0.333*(V24)$) {};
\draw [color=black] ($(24)$) -- ($(S2)$);
\draw [color=black] ($(23)$) -- ($(S2)$);
\draw [color=black] ($0.4*(V234) + 0.6*(S2)$) -- ($(S2)$);
\draw [dashed,color=black] ($0.7*(V234) + 0.3*(S2)$) -- ($(S2)$);
\draw [color=black] ($0.6*(234) + 0.4*(VS2)$) -- ($(234)$);
\draw [dashed,color=black] ($0.3*(234) + 0.7*(VS2)$) -- ($(234)$);

\node (V324) at (3.3757,2.173) {};
\node (V32) at (-5.0161,-1.18829) {};
\node (V34) at (-3.9885,-3.2657) {};

\node (S3) [star,fill,inner sep=2pt, label=225:{$\mathrm{S}^{3}$}] at ($0.333*(V324) + 0.333*(32) + 0.333*(34)$) {};
\node (VS3) at ($0.333*(324) + 0.333*(V32) + 0.333*(V34)$) {};
\draw [color=black] ($(34)$) -- ($(S3)$);
\draw [color=black] ($(32)$) -- ($(S3)$);
\draw [color=black] ($0.4*(V324) + 0.6*(S3)$) -- ($(S3)$);
\draw [dashed,color=black] ($0.7*(V324) + 0.3*(S3)$) -- ($(S3)$);
\draw [color=black] ($0.6*(324) + 0.4*(VS3)$) -- ($(324)$);
\draw [dashed,color=black] ($0.3*(324) + 0.7*(VS3)$) -- ($(324)$);

\node (V423) at (0,-4.1212) {};
\node (V42) at (-1.2542,5.29374) {};
\node (V43) at (1.2542,5.29374) {};

\node (S4) [star,fill,inner sep=2pt,label=90:{$\mathrm{S}^{4}$}] at ($0.333*(V423) + 0.333*(43) + 0.333*(42)$) {};
\node (VS4) at ($0.333*(423) + 0.333*(V43) + 0.333*(V42)$) {};
\draw [color=black] ($(42)$) -- ($(S4)$);
\draw [color=black] ($(43)$) -- ($(S4)$);
\draw [color=black] ($0.4*(V423) + 0.6*(S4)$) -- ($(S4)$);
\draw [dashed,color=black] ($0.7*(V423) + 0.3*(S4)$) -- ($(S4)$);
\draw [color=black] ($0.6*(423) + 0.4*(VS4)$) -- ($(423)$);
\draw [dashed,color=black] ($0.3*(423) + 0.7*(VS4)$) -- ($(423)$);

\node [regular polygon,regular polygon sides=3,fill,color=Red,inner sep=2pt] at ($(14)$) {};
\node [regular polygon,regular polygon sides=3,fill,color=Red,inner sep=2pt] at ($(13)$) {};
\node [regular polygon,regular polygon sides=3,fill,color=Red,inner sep=2pt] at ($(134)$) {};
\node [regular polygon,regular polygon sides=3,fill,color=Red,inner sep=2pt] at ($(23)$) {};
\node [regular polygon,regular polygon sides=3,fill,color=Red,inner sep=2pt] at ($(24)$) {};
\node [regular polygon,regular polygon sides=3,fill,color=Red,inner sep=2pt] at ($(234)$) {};
\node [regular polygon,regular polygon sides=3,fill,color=Red,inner sep=2pt] at ($(32)$) {};
\node [regular polygon,regular polygon sides=3,fill,color=Red,inner sep=2pt] at ($(34)$) {};
\node [regular polygon,regular polygon sides=3,fill,color=Red,inner sep=2pt] at ($(324)$) {};
\node [regular polygon,regular polygon sides=3,fill,color=Red,inner sep=2pt] at ($(42)$) {};
\node [regular polygon,regular polygon sides=3,fill,color=Red,inner sep=2pt] at ($(43)$) {};
\node [regular polygon,regular polygon sides=3,fill,color=Red,inner sep=2pt] at ($(423)$) {};

\node [diamond,fill,color=ForestGreen,inner sep=2.4pt] at ($(1423)$) {};
\node [diamond,fill,color=ForestGreen,inner sep=2.4pt] at ($(1432)$) {};
\node [diamond,fill,color=ForestGreen,inner sep=2.4pt] at ($(32423)$) {};
\node [diamond,fill,color=ForestGreen,inner sep=2.4pt] at ($(23423)$) {};
\node [diamond,fill,color=ForestGreen,inner sep=2.4pt] at ($(1324)$) {};
\node [diamond,fill,color=ForestGreen,inner sep=2.4pt] at ($(24324)$) {};
\node [diamond,fill,color=ForestGreen,inner sep=2.4pt] at ($(32442)$) {};
\node [diamond,fill,color=ForestGreen,inner sep=2.4pt] at ($(1342)$) {};
\node [diamond,fill,color=ForestGreen,inner sep=2.4pt] at ($(13443)$) {};
\node [diamond,fill,color=ForestGreen,inner sep=2.4pt] at ($(23443)$) {};
\node [diamond,fill,color=ForestGreen,inner sep=2.4pt] at ($(23434)$) {};
\node [diamond,fill,color=ForestGreen,inner sep=2.4pt] at ($(13434)$) {};

\node [star,fill,color=blue!80!white,inner sep=2.2pt] at ($(S1)$) {};
\node [star,fill,color=blue!80!white,inner sep=2.2pt] at ($(S2)$) {};
\node [star,fill,color=blue!80!white,inner sep=2.2pt] at ($(S3)$) {};
\node [star,fill,color=blue!80!white,inner sep=2.2pt] at ($(S4)$) {};

\end{tikzpicture}

\caption[The Link of $\Theta^0$]{The link of $\left[\textrm{id},\Theta^0\right]$ in $\K_4$. All of the adjacent vertices are labeled with $\mathrm{id} \in \Out^0(W_4)$, so this is the same as the link of $\Theta^0$ in $\HT_4$. The vertices are labeled with their corresponding hypertree. The blue stars are star trees, the green diamonds are line trees, and the red triangles are hypertrees in $\mathcal{M}_4^1$. Dashed lines connect to the other side of the link. (See Piggott \cite{Piggott} for another picture.)}
\label{fig:lknuc}
\end{figure}

This is enough information to show that no angle solutions are possible.

\begin{theorem}
\label{thm:noout}
There does not exist an $\Out(W_4)$-equivariant piecewise Euclidean (or piecewise hyperbolic) $\cat(0)$ $\LP\cat(-1)\RP$ metric on $\K_4$. 
\end{theorem}

\begin{proof}
If there did exist such a metric, then by Theorem  \ref{thm:brid}, there would exist angles $\alpha_L,\alpha_S,\beta_L,\beta_S,\gamma_L, \gamma_S \in (0,\pi]$ that satisfy Inequalities \eqref{eq:TL} - \eqref{eq:alas} above. Let us show that these are inconsistent. 

\begin{align*}
\pi \geq \alpha_L + \beta_L + \gamma_L &\geq \frac{\pi}{4} + \frac{\pi}{2} + \frac{\pi}{4} = \pi &\textrm{by \eqref{eq:TL}, \eqref{eq:al}, \eqref{eq:bl}, \eqref{eq:gl}} \\
\implies \alpha_L + \beta_L + \gamma_L &= \pi \numberthis \label{eq:alblgl} \\
\implies \alpha_L = \pi - \beta_L - \gamma_L &\leq \pi - \frac{\pi}{2} - \frac{\pi}{4} = \frac{\pi}{4} \leq \alpha_L &\textrm{by \eqref{eq:bl}, \eqref{eq:gl}, \eqref{eq:al}} \\
\implies \alpha_L &= \frac{\pi}{4} \numberthis \label{eq:alf} \\ \\
\end{align*}
\begin{align*}
\beta_L = \pi - \alpha_L - \gamma_L &= \frac{3\pi}{4} - \gamma_L \leq \frac{3\pi}{4} - \frac{\pi}{4} = \frac{\pi}{2} \leq \beta_L &\textrm{by \eqref{eq:alblgl}, \eqref{eq:alf}, \eqref{eq:gl}, \eqref{eq:bl}} \\
\implies \beta_L &= \frac{\pi}{2}  \numberthis \label{eq:blf} \\ \\
\end{align*}
\begin{align*}
\gamma_L = \pi - \alpha_L - \beta_L &= \pi - \frac{\pi}{4} - \frac{\pi}{2} = \frac{\pi}{4} &\textrm{by \eqref{eq:alblgl}, \eqref{eq:alf}, \eqref{eq:blf}} \\ 
\implies \gamma_L &= \frac{\pi}{4} \numberthis \label{eq:glf} \\ \\
\end{align*}
\begin{align*}
\alpha_S \geq \frac{\pi}{2} - \alpha_L &= \frac{\pi}{2} - \frac{\pi}{4} = \frac{\pi}{4} &\textrm{by \eqref{eq:alas}, \eqref{eq:alf}} \\
\implies \alpha_S &\geq \frac{\pi}{4} \numberthis \label{eq:asf} \\ \\
\end{align*}
\begin{align*}
\beta_S \geq \pi - \beta_L &= \pi - \frac{\pi}{2} = \frac{\pi}{2} &\textrm{by \eqref{eq:blbs}, \eqref{eq:blf}} \\
\implies \beta_S &\geq \frac{\pi}{2} \numberthis \label{eq:bsf} \\ \\
\end{align*}
\begin{align*}
\gamma_S \leq \pi - \alpha_S - \beta_S &\leq \pi - \frac{\pi}{4} - \frac{\pi}{2} = \frac{\pi}{4} &\textrm{by \eqref{eq:TS}, \eqref{eq:asf}, \eqref{eq:bsf}} \\
\implies \gamma_S &\leq \frac{\pi}{4}  \numberthis \label{eq:gsf} \\ \\
\end{align*}
\begin{align*}
\therefore \ \  \frac{\pi}{3} \leq \gamma_S &\leq \frac{\pi}{4} &\textrm{by \eqref{eq:gs}, \eqref{eq:gsf}} 
\end{align*}

This is a contradiction, and so we are done. 

\end{proof}

\subsection{The $\textrm{Out}^0(W_4)$ Case}
\label{subsec:out0}

Since being a $\cat(\kappa)$ group is not a property that is in general preserved under finite extension, it is possible that $\Out^0(W_4)$ is a $\cat(\kappa)$ group, but $\Out(W_4)$ is not. So while $\K_4$ could not be made into a $\cat(\kappa)$ $M_{\kappa}$-simplicial complex that was equivariant with respect to the full $\Out(W_4)$ action, it is \emph{a priori} possible that we could relax the requirement and obtain a metric only equivariant with respect to the induced $\Out^0(W_4)$ action. It turns out that this is still impossible.

In Subsection \ref{subsec:out}, the quotient of $\K_4$ by $\Out(W_4)$ consisted of two simplices, and so only eight angle variables were necessary to consider. On the other hand, the quotient of $\K_4$ by $\Out(W_4)$ is a full copy of $\HT_4$, which consists of 24 L-simplices and 12 S-simplices, for a total of 36 simplices and so 108 angles. Our number of inequalities will rise as well. For instance, there will be 24 of type \eqref{eq:TL}, 12 of type \eqref{eq:TS}, and so on. There will even be additional forms of inequalities such as $\beta_{L^{i,j}_{k,l}} + \beta_{L^{i,j}_{k,l}} + \beta_{L^{i,j}_{l,k}} + \beta_{L^{i,j}_{l,k}} \geq 2 \pi$, since in the link of $\left[\textrm{id},\Omega^{i,j}\right]$, the vertex angles connecting to the different line graphs could now be different. So our direct approach in Theorem \ref{thm:noout} is too cumbersome to try again identically. Instead, we'll use the additional $\Sigma_4$ symmetry in the quotient $\HT_4$ to simplify the calculations and prove the following theorem. 

\begin{theorem}
\label{thm:noout0}
There does not exist an $\Out^0(W_4)$-equivariant piecewise Euclidean (or piecewise hyperbolic) $\cat(0)$ $\LP\cat(-1)\RP$ metric on $\K_4$. 
\end{theorem}

First, we need to find a convenient way to name these 108 variables and describe their inequalities.

\begin{definition}
\label{def:ang0}
Suppose $K_4$ has been given an $\Out^0(W_4)$-equivariant metric to turn it into a $M_{\kappa}$-polyhedral complex. Since the metric is $\Out^0(W_4)$-equivariant, it suffices to assign an angle to each corner of each simplex in the fundamental domain of the action in order to specify an angle in every corner of every simplex of $\K_4$. The fundamental domain is isometric to the quotient $\HT_4$. So let the angles be defined as follows:
\begin{enumerate}
\item In any L-simplex, there are vertices of the form $\left[\alpha,\Theta^0 \right]$, $\left[\alpha,L^{i,j}_{k,l} \right]$, and either $\left[\alpha,\Omega^{i,j} \right]$ or $\left[\alpha,\Omega^{k,l} \right]$. Since $L^{i,j}_{k,l}$ is the same labeled hypertree as $L^{k,l}_{i,j}$, we usually restrict the indexing to $i < k$, i.e., the smaller of the two is in the superscript. However, in the L-simplex, we also want to keep track of which $\Omega$ vertex is present. So we will subscript the angles in this simplex with $L^{i,j}_{k,l}$ where the $\{i,j\}$ superscript indicates which $\Omega^{i,j}$ is present. So for 
instance, $\alpha_{L^{1,3}_{2,4}}$ will be the vertex angle of $\left[\alpha,\Theta^0 \right]$, $\beta_{L^{1,3}_{2,4}}$ will be the vertex angle of $\left[\alpha,\Omega^{1,3}\right]$, and 
$\gamma_{L^{1,3}_{2,4}}$ will be the vertex angle of 
$\left[\alpha,L^{1,3}_{2,4}\right]$. On the other hand, 
$\alpha_{L^{2,4}_{1,3}}$ will be the vertex angle of 
$\left[\alpha,\Theta^0 \right]$, $\beta_{L^{2,4}_{1,3}}$ will be the vertex angle of $\left[\alpha,\Omega^{2,4}\right]$, and 
$\gamma_{L^{2,4}_{1,3}}$ 
will be the vertex angle of $\left[\alpha,L^{2,4}_{1,3}\right] = \left[\alpha,L^{1,3}_{2,4}\right]$.
\item In any S-simplex, there are vertices of the form $\left[\alpha,\Theta^0 \right]$,  $\left[\alpha,\Omega^{i,j} \right]$, and $\left[\alpha,S^{i} \right]$. The indexing is much easier here, since adding the $\{i,j\}$ superscript uniquely specifies the star tree. So for instance, $\alpha_{S^{1,3}}$ will be the vertex angle of $\left[\alpha,\Theta^0\right]$, $\beta_{S^{1,3}}$ will be the vertex angle of $\left[\alpha,\Omega^{1,3}\right]$, and $\gamma_{S^{1,3}}$ will be the vertex angle of $\left[\alpha,S^{1}\right]$.
\end{enumerate}
\end{definition}

\begin{notation*}
Throughout the rest of this section, we adopt the convention that when indexes $i$, $j$, $k$, and $l$ appear in subscripts and superscripts of the hypertree or angle notation, it is assumed that the indexes are drawn from $[4]$, are distinct, and that the listed inequalities hold for all such choices of the indices. 
\end{notation*}

By Theorem \ref{thm:brid}, we get these inequalities for each simplex in $\HT_4$:

\begin{align}\label{eq:TL0}
\alpha_{L^{i,j}_{k,l}} + \beta_{L^{i,j}_{k,l}} + \gamma_{L^{i,j}_{k,l}} \leq \pi
\end{align}

\begin{align}\label{eq:TS0}
\alpha_{S^{i,j}} + \beta_{S^{i,j}} + \gamma_{S^{i,j}} \leq \pi
\end{align}

Now we need to re-examine injective loops in the links of vertices in $\K_4$ to find appropriate inequalities. All of the links of vertices look identical to the links is Subsection \ref{subsec:out} except that the angle labels now have (possibly different) indices. These indices are determined by the indices of the adjacent hypertrees but not the $\Out^0(W_4)$ label. See Figure \ref{fig:lkomeg0}.

\begin{figure}
\hspace*{-1.5cm}
\begin{tikzpicture}[scale=1.5]
\centering

\node (S1) [star,fill,inner sep=2pt,label=0:{$\left[\mathrm{id},\mathrm{S}^{1}\right]$}] at (0, 0) {};
\node (S1L) at ($(S1)+(0.75,-0.4)$) {$= \left[x_{1,\{3\}},\mathrm{S}^{1}\right]$};
\node (1324) [circle,fill,inner sep=1.5pt,label=0:{$\left[\mathrm{id},\mathrm{L}^{1,3}_{2,4}\right] = \left[x_{1,\{3\}},\mathrm{L}^{1,2}_{3,4}\right]$}] at (2, 0) {};
\node (1342) [circle,fill,inner sep=1.5pt,label=180:{$\left[\mathrm{id},\mathrm{L}^{1,3}_{4,2}\right] = \left[x_{1,\{3\}},\mathrm{L}^{1,3}_{4,2}\right]$}] at (-2, 0) {};
\node (T1) [circle,fill,inner sep=1.5pt,label=-90:{$\left[\mathrm{id},\Theta^0\right]$}] at (0, -2) {};
\node (T2) [circle,fill,inner sep=1.5pt,label=90:{$\left[x_{1,\{3\}},\Theta^0\right]$}] at (0, 2) {};

\draw [color=black] ($(T1)$) -- ($(S1)$);
\draw [color=black] ($(T1)$) -- ($(1324)$);
\draw [color=black] ($(T1)$) -- ($(1342)$);
\draw [color=black] ($(T2)$) -- ($(S1)$);
\draw [color=black] ($(T2)$) -- ($(1324)$);
\draw [color=black] ($(T2)$) -- ($(1342)$);

\node [diamond,fill,color=ForestGreen,inner sep=2.4pt] at ($(1324)$) {};
\node [diamond,fill,color=ForestGreen,inner sep=2.4pt] at ($(1342)$) {};

\node [circle,fill,color=Plum,inner sep=2.6pt] at ($(T1)$) {};
\node [circle,fill,color=Plum,inner sep=2.6pt] at ($(T2)$) {};

\node [star,fill,color=blue!80!white,inner sep=3pt] at ($(S1)$) {};

\node at ($0.5*(T1)+ 0.5*(1342) + (-0.2,-0.2)$) {$\beta_{\textrm{L}^{1,3}_{4,2}}$};
\node at ($0.5*(T2)+ 0.5*(1342) + (-0.2,0.2)$) {$\beta_{\textrm{L}^{1,3}_{4,2}}$};
\node at ($0.5*(T2)+ 0.5*(1324) + (0.2,0.2)$) {$\beta_{\textrm{L}^{1,3}_{2,4}}$};
\node at ($0.5*(T1)+ 0.5*(1324) + (0.2,-0.3)$) {$\beta_{\textrm{L}^{1,3}_{2,4}}$};
\node at ($0.5*(T1)+ 0.5*(S1) + (-0.2,0)$) {$\beta_{\textrm{S}^{1}}$};
\node at ($0.5*(T2)+ 0.5*(S1) + (-0.2,0)$) {$\beta_{\textrm{S}^{1}}$};

\end{tikzpicture}

\caption[The Link of $\Omega^{1,3}$ with $\Out^0(W_4)$-Equivariant Angles]{Another picture of the link of $\left[\textrm{id},\Omega^{1,3}\right]$ in $\K_4$ with angle variables eqivariant with respect to the action of $\Out^0(W_4)$. For the $\Out(W_4)$ case, the picture is the same but we can ignore the indexing on the angles.}
\label{fig:lkomeg0}
\end{figure}
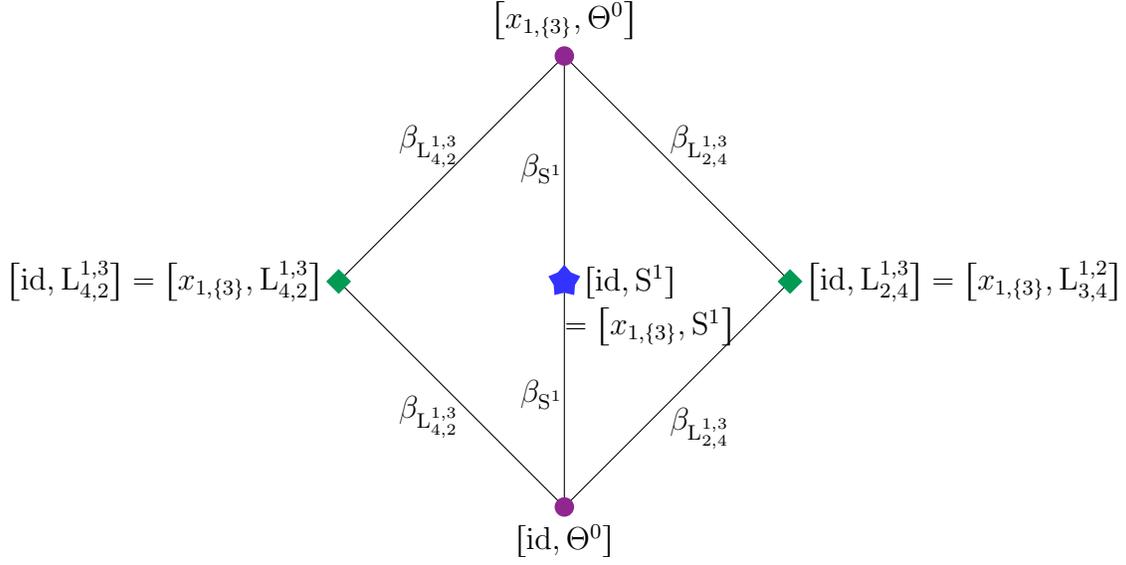

\begin{align}
\label{eq:bl0}
\beta_{L^{i,j}_{k,l}} + \beta_{L^{i,j}_{k,l}} + \beta_{L^{i,j}_{l,k}} + \beta_{L^{i,j}_{l,k}} &\geq 2 \pi \\ 
\textrm{i.e., }\beta_{L^{i,j}_{k,l}} + \beta_{L^{i,j}_{l,k}} &\geq \pi \nonumber
\end{align}

\begin{align}
\label{eq:blbs0}
\beta_{L^{i,j}_{k,l}} + \beta_{L^{i,j}_{k,l}} + \beta_{S^{i,j}} + \beta_{S^{i,j}} &\geq 2 \pi
\\ 
\textrm{i.e., }\beta_{L^{i,j}_{k,l}} + \beta_{S^{i,j}} &\geq \pi \nonumber
\end{align}

Notice that $\beta_{L^{i,j}_{l,k}} + \beta_{S^{i,j}} \geq \pi$, which is also an inequality derivable from that link, is included in Inequality \eqref{eq:blbs0} since our notation implicitly quantifies over the different possibilities for $k$ and $l$. 

We continue to examine injective loops in the links of vertices. 

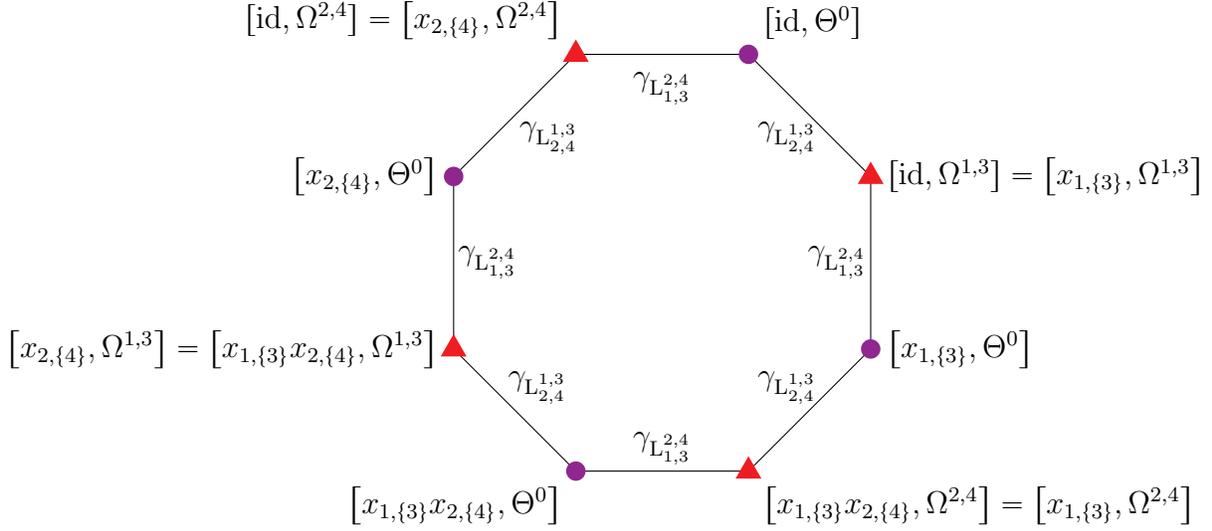
\begin{figure}
\hspace*{-2cm}
\begin{tikzpicture}[scale=1.5]
\centering

\node (13I) [circle,fill,inner sep=1.5pt,label=0:{$\left[\mathrm{id},\Omega^{1,3}\right] = \left[x_{1,\{3\}},\Omega^{1,3}\right]$}] at (1.848, 0.765) {};
\node (TI) [circle,fill,inner sep=1.5pt,label=45:{$\left[\mathrm{id},\Theta^0 \right]$}] at (0.765, 1.848) {};
\node (24I) [circle,fill,inner sep=1.5pt,label=135:{$\left[\mathrm{id},\Omega^{2,4}\right] = \left[x_{2,\{4\}},\Omega^{2,4}\right]$}] at (-0.765, 1.848) {};
\node (T24) [circle,fill,inner sep=1.5pt,label=180:{$\left[x_{2,\{4\}},\Theta^0\right]$}] at (-1.848, 0.765) {};
\node (1324) [circle,fill,inner sep=1.5pt,label=180:{$\left[x_{2,\{4\}},\Omega^{1,3}\right] = \left[x_{1,\{3\}}x_{2,\{4\}},\Omega^{1,3}\right]$}] at (-1.848, -0.765) {};
\node (T1324) [circle,fill,inner sep=1.5pt,label=225:{$\left[x_{1,\{3\}}x_{2,\{4\}},\Theta^0\right]$}] at (-0.765, -1.848) {};
\node (2413) [circle,fill,inner sep=1.5pt,label=-45:{$\left[x_{1,\{3\}}x_{2,\{4\}},\Omega^{2,4}\right] = \left[x_{1,\{3\}},\Omega^{2,4}\right]$}] at (0.765, -1.848) {};
\node (T13) [circle,fill,inner sep=1.5pt,label=0:{$\left[x_{1,\{3\}},\Theta^0\right]$}] at (1.848, -0.765) {};

\draw [color=black] ($(13I)$) -- ($(TI)$);
\draw [color=black] ($(24I)$) -- ($(TI)$);
\draw [color=black] ($(24I)$) -- ($(T24)$);
\draw [color=black] ($(1324)$) -- ($(T24)$);
\draw [color=black] ($(1324)$) -- ($(T1324)$);
\draw [color=black] ($(2413)$) -- ($(T1324)$);
\draw [color=black] ($(2413)$) -- ($(T13)$);
\draw [color=black] ($(13I)$) -- ($(T13)$);

\node [regular polygon,regular polygon sides=3,fill,color=Red,inner sep=2pt] at ($(13I)$) {};
\node [regular polygon,regular polygon sides=3,fill,color=Red,inner sep=2pt] at ($(24I)$) {};
\node [regular polygon,regular polygon sides=3,fill,color=Red,inner sep=2pt] at ($(1324)$) {};
\node [regular polygon,regular polygon sides=3,fill,color=Red,inner sep=2pt] at ($(2413)$) {};

\node [circle,fill,color=Plum,inner sep=2.6pt] at ($(TI)$) {};
\node [circle,fill,color=Plum,inner sep=2.6pt] at ($(T24)$) {};
\node [circle,fill,color=Plum,inner sep=2.6pt] at ($(T1324)$) {};
\node [circle,fill,color=Plum,inner sep=2.6pt] at ($(T13)$) {};

\node at ($0.5*(13I)+ 0.5*(TI) + (-0.2,-0.2)$) {$\gamma_{\textrm{L}^{1,3}_{2,4}}$};
\node at ($0.5*(24I)+ 0.5*(TI) + (0,-0.3)$) {$\gamma_{\textrm{L}^{2,4}_{1,3}}$};
\node at ($0.5*(24I)+ 0.5*(T24) + (0.3,-0.2)$) {$\gamma_{\textrm{L}^{1,3}_{2,4}}$};
\node at ($0.5*(1324)+ 0.5*(T24) + (0.3,0)$) {$\gamma_{\textrm{L}^{2,4}_{1,3}}$};
\node at ($0.5*(1324)+ 0.5*(T1324) + (0.2,0.2)$) {$\gamma_{\textrm{L}^{1,3}_{2,4}}$};
\node at ($0.5*(2413)+ 0.5*(T1324) + (0,0.2)$) {$\gamma_{\textrm{L}^{2,4}_{1,3}}$};
\node at ($0.5*(2413)+ 0.5*(T13) + (-0.2,0.2)$) {$\gamma_{\textrm{L}^{1,3}_{2,4}}$};
\node at ($0.5*(13I)+ 0.5*(T13) + (-0.3,0)$) {$\gamma_{\textrm{L}^{2,4}_{1,3}}$};

\end{tikzpicture}
\caption[The Link of $\textrm{L}^{1,3}_{2,4}$ with $\Out^0(W_4)$-Equivariant Angles]{Another picture of the link of $\left[\textrm{id},\textrm{L}^{1,3}_{2,4}\right]$ in $\K_4$ with angle variables eqivariant with respect to the action of $\Out^0(W_4)$. For the $\Out(W_4)$ case, the picture is the same but we can ignore the indexing on the angles.}
\label{fig:lkline0}
\end{figure}

See Figure \ref{fig:lkline0}. 

\begin{align}
\label{eq:gl0}
4 \gamma_{L^{i,j}_{k,l}} + 4 \gamma_{L^{k,l}_{i,j}}  &\geq 2 \pi \\ 
\textrm{i.e., }\gamma_{L^{i,j}_{k,l}} + \gamma_{L^{k,l}_{i,j}} &\geq \frac{\pi}{2} \nonumber
\end{align}

\begin{figure}
\hspace*{-2cm}
\begin{tikzpicture}[scale=2]
\centering

\node (14I) [circle,fill,inner sep=1.5pt,label=0:{$\left[\mathrm{id},\Omega^{1,4}\right]$}] at (0,0) {};

\node (14IL) at ($(14I)+(0.9,-0.3)$) {$= \left[x_{1,\{4\}},\Omega^{1,4}\right]$};

\node (TI) [circle,fill,inner sep=1.5pt,label=90:{$\left[\mathrm{id},\Theta^0 \right]$}] at (0, 2) {};
\node (13I) [circle,fill,inner sep=1.5pt,label=180:{$\left[\mathrm{id},\Omega^{1,3}\right] = \left[x_{1,\{3\}},\Omega^{1,3}\right]$}] at (-2,2) {};
\node (T13) [circle,fill,inner sep=1.5pt,label=0:{$\left[x_{1,\{3\}},\Theta^0 \right]$}] at (-2, 0) {};
\node (1213) [circle,fill,inner sep=1.5pt,label=-90:{$\left[x_{1,\{3\}},\Omega^{1,2}\right]$}] at (-2,-2) {};

\node (1213L) at ($(1213)+(0.5,-0.5)$) {$= \left[x_{1,\{4\}},\Omega^{1,2}\right]$};

\node (T14) [circle,fill,inner sep=1.5pt,label=-90:{$\left[x_{1,\{4\}},\Theta^0 \right]$}] at (0, -2) {};
\node (1314) [circle,fill,inner sep=1.5pt,label=-90:{$\left[x_{1,\{4\}},\Omega^{1,3}\right]$}] at (2,-2) {};

\node (1314L) at ($(1314)+(-0.3,-0.5)$) {$= \left[x_{1,\{3,4\}},\Omega^{1,3}\right]$};

\node (T134) [circle,fill,inner sep=1.5pt,label=45:{$\left[x_{1,\{3,4\}},\Theta^0 \right]$}] at (2, 0) {};
\node (12I) [circle,fill,inner sep=1.5pt,label=0:{$\left[\textrm{id},\Omega^{1,2}\right] = \left[x_{1,\{3,4\}},\Omega^{1,2}\right]$}] at (2,2) {};
\node (1413) [circle,fill,inner sep=1.5pt,label=-90:{$\left[x_{1,\{3\}},\Omega^{1,4}\right] = \left[x_{1,\{3,4\}},\Omega^{1,4}\right]$}] at (0,-4) {};

\draw [color=black] ($(14I)$) -- ($(TI)$);
\draw [color=black] ($(13I)$) -- ($(TI)$);
\draw [color=black] ($(13I)$) -- ($(T13)$);
\draw [color=black] ($(1213)$) -- ($(T13)$);
\draw [color=black] ($(1213)$) -- ($(T14)$);
\draw [color=black] ($(1314)$) -- ($(T14)$);
\draw [color=black] ($(1314)$) -- ($(T134)$);
\draw [color=black] ($(12I)$) -- ($(T134)$);
\draw [color=black] ($(12I)$) -- ($(TI)$);
\draw [color=black] ($(14I)$) -- ($(T14)$);

\draw [color=black] ($(T13)$) .. controls (-4,-2) .. ($(1413)$);
\draw [color=black] ($(T134)$) .. controls (4,-2) .. ($(1413)$);

\node [regular polygon,regular polygon sides=3,fill,color=Red,inner sep=2pt] at ($(14I)$) {};
\node [regular polygon,regular polygon sides=3,fill,color=Red,inner sep=2pt] at ($(13I)$) {};
\node [regular polygon,regular polygon sides=3,fill,color=Red,inner sep=2pt] at ($(1213)$) {};
\node [regular polygon,regular polygon sides=3,fill,color=Red,inner sep=2pt] at ($(1314)$) {};
\node [regular polygon,regular polygon sides=3,fill,color=Red,inner sep=2pt] at ($(12I)$) {};
\node [regular polygon,regular polygon sides=3,fill,color=Red,inner sep=2pt] at ($(1413)$) {};

\node [circle,fill,color=Plum,inner sep=2.6pt] at ($(TI)$) {};
\node [circle,fill,color=Plum,inner sep=2.6pt] at ($(T13)$) {};
\node [circle,fill,color=Plum,inner sep=2.6pt] at ($(T134)$) {};
\node [circle,fill,color=Plum,inner sep=2.6pt] at ($(T14)$) {};

\node at ($0.5*(14I)+ 0.5*(TI) + (0.3,0)$) {$\gamma_{\textrm{S}^{1,4}}$};
\node at ($0.5*(13I)+ 0.5*(TI) + (0,0.2)$) {$\gamma_{\textrm{S}^{1,3}}$};
\node at ($0.5*(13I)+ 0.5*(T13) + (-0.3,0)$) {$\gamma_{\textrm{S}^{1,3}}$};
\node at ($0.5*(1213)+ 0.5*(T13) + (-0.3,0)$) {$\gamma_{\textrm{S}^{1,2}}$};
\node at ($0.5*(1213)+ 0.5*(T14) + (0,0.2)$) {$\gamma_{\textrm{S}^{1,2}}$};
\node at ($0.5*(1314)+ 0.5*(T14) + (0,0.2)$) {$\gamma_{\textrm{S}^{1,3}}$};
\node at ($0.5*(1314)+ 0.5*(T134) + (0.3,0)$) {$\gamma_{\textrm{S}^{1,3}}$};
\node at ($0.5*(12I)+ 0.5*(T134) + (0.3,0)$) {$\gamma_{\textrm{S}^{1,2}}$};
\node at ($0.5*(12I)+ 0.5*(TI) + (0,0.2)$) {$\gamma_{\textrm{S}^{1,2}}$};
\node at ($0.5*(14I)+ 0.5*(T14) + (0.3,0)$) {$\gamma_{\textrm{S}^{1,4}}$};

\node at ($(-2.3,-3)$) {$\gamma_{\textrm{S}^{1,4}}$};
\node at ($(2.3,-3)$) {$\gamma_{\textrm{S}^{1,4}}$};

\end{tikzpicture}

\caption[The Link of $\textrm{S}^{1}$ with $\Out^0(W_4)$-Equivariant Angles]{Another picture of the link of $\left[\textrm{id},\textrm{S}^{1}\right]$ in $\K_4$ with angle variables eqivariant with respect to the action of $\Out^0(W_4)$. For the $\Out(W_4)$ case, the picture is the same but we can ignore the indexing on the angles.}
\label{fig:lkstar0}
\end{figure}
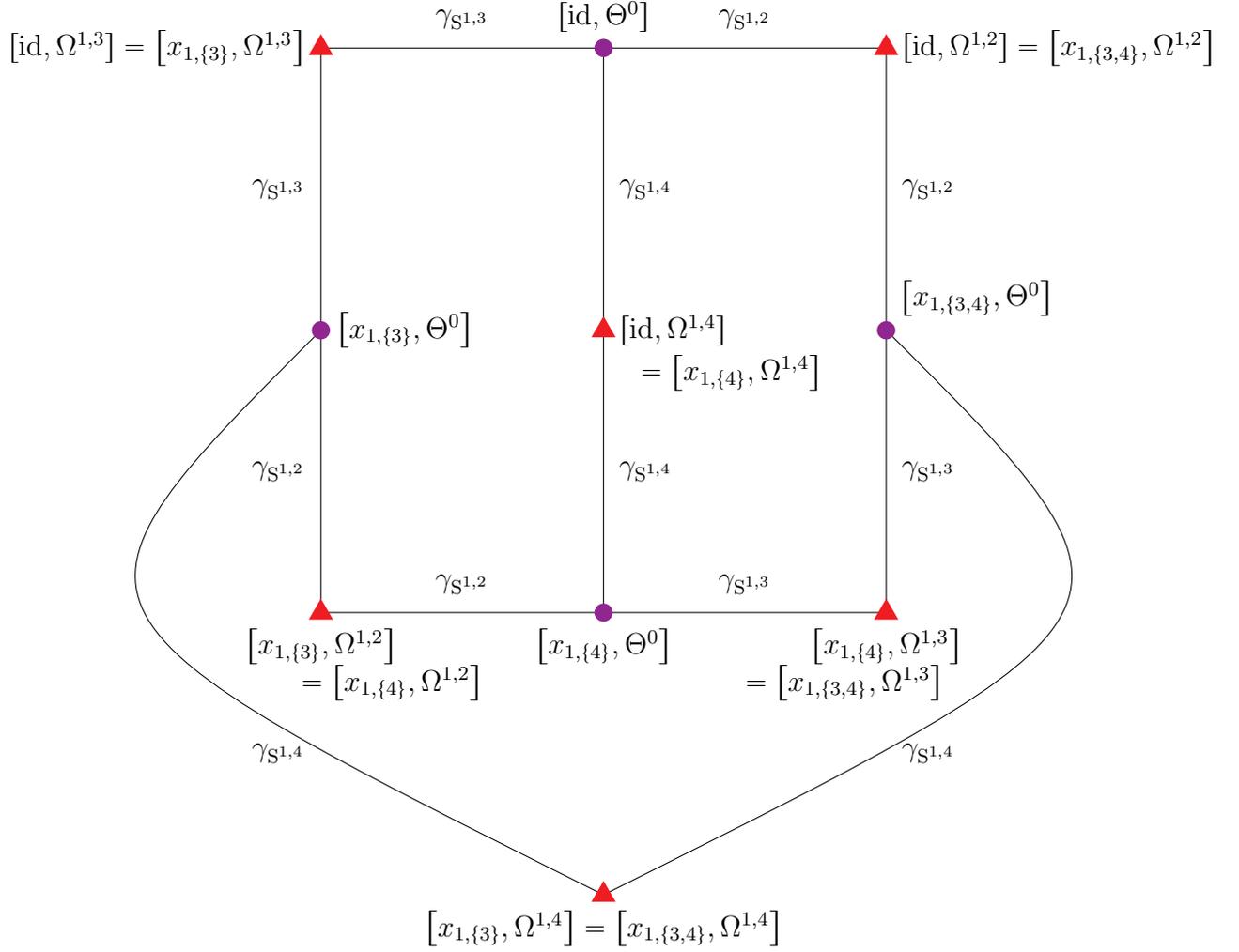

See Figure \ref{fig:lkstar0}.

\begin{align}
\label{eq:gs0}
2 \gamma_{S^{i,j}} + 2 \gamma_{S^{i,k}} + 2 \gamma_{S^{i,l}} &\geq 2 \pi \\ 
\textrm{i.e., }\gamma_{S^{i,j}} + \gamma_{S^{i,k}} + \gamma_{S^{i,l}} &\geq \pi \nonumber
\end{align}

\begin{figure}
\hspace*{-3.5cm}%
\begin{tikzpicture}[scale=2]
\centering

\node (14) [circle,fill,inner sep=1.5pt,label=0:{$\Omega^{1,4}$}] at (0, 1) {};
\node (13) [circle,fill,inner sep=1.5pt,label=-5:{$\Omega^{1,3}$}] at (-0.866, -0.5) {};
\node (134) [circle,fill,inner sep=1.5pt,label=225:{$\Omega^{1,2}$}] at (0.866, -0.5) {};
\node (S1) [star,fill,inner sep=2pt,label=45:{$\mathrm{S}^{1}$}] at (0, 0) {};
\draw [color=black] ($(14)$) -- ($(S1)$);
\draw [color=black] ($(13)$) -- ($(S1)$);
\draw [color=black] ($(134)$) -- ($(S1)$);

\node (32) [circle,fill,inner sep=1.5pt,label=45:{$\Omega^{3,2}$}] at (1.0607, 2.0607) {};
\node (1432) [circle,fill,inner sep=1.5pt,label=0:{$\mathrm{L}^{1,4}_{3,2}$}] at ($0.5*(14) + 0.5*(32)$) {};
\draw [color=black] ($(14)$) -- ($(32)$) -- ($(1432)$) -- ($(14)$) -- cycle;

\node (23) [circle,fill,inner sep=1.5pt,label=135:{$\Omega^{2,3}$}] at (-1.0607, 2.0607) {};
\node (1423) [circle,fill,inner sep=1.5pt,label=200:{$\mathrm{L}^{1,4}_{2,3}$}] at ($0.5*(14) + 0.5*(23)$) {};
\draw [color=black] ($(14)$) -- ($(23)$) -- ($(1423)$) -- ($(14)$) -- cycle;

\node (423) [circle,fill,inner sep=1.5pt,label=0:{$\Omega^{4,1}$}] at (0, 3.1214) {};
\node (23423) [circle,fill,inner sep=1.5pt,label=180:{$\mathrm{L}^{2,3}_{4,1}$}] at ($0.5*(23) + 0.5*(423)$) {};
\node (32423) [circle,fill,inner sep=1.5pt,label=0:{$\mathrm{L}^{3,2}_{4,1}$}] at ($0.5*(32) + 0.5*(423)$) {};
\draw [color=black] ($(423)$) -- ($(23)$) -- ($(23423)$) -- ($(423)$) -- cycle;
\draw [color=black] ($(423)$) -- ($(32)$) -- ($(32423)$) -- ($(423)$) -- cycle;

\node (24) [circle,fill,inner sep=1.5pt,label=180:{$\Omega^{2,4}$}] at (-2.315, 0.1118) {};
\node (1324) [circle,fill,inner sep=1.5pt,label=90:{$\mathrm{L}^{1,3}_{2,4}$}] at ($0.5*(13) + 0.5*(24)$) {};
\draw [color=black] ($(13)$) -- ($(24)$) -- ($(1324)$) -- ($(13)$) -- cycle;

\node (42) [circle,fill,inner sep=1.5pt,label=270:{$\Omega^{4,2}$}] at (-1.2542, -1.94889) {};
\node (1342) [circle,fill,inner sep=1.5pt,label=0:{$\mathrm{L}^{1,3}_{4,2}$}] at ($0.5*(13) + 0.5*(42)$) {};
\draw [color=black] ($(13)$) -- ($(42)$) -- ($(1342)$) -- ($(13)$) -- cycle;

\node (324) [circle,fill,inner sep=1.5pt,label=270:{$\Omega^{3,1}$}] at (-2.7032, -1.3371) {};
\node (24324) [circle,fill,inner sep=1.5pt,label=180:{$\mathrm{L}^{2,4}_{3,1}$}] at ($0.5*(24) + 0.5*(324)$) {};
\node (32442) [circle,fill,inner sep=1.5pt,label=270:{$\mathrm{L}^{3,1}_{4,2}$}] at ($0.5*(324) + 0.5*(42)$) {};
\draw [color=black] ($(324)$) -- ($(24)$) -- ($(24324)$) -- ($(324)$) -- cycle;
\draw [color=black] ($(324)$) -- ($(42)$) -- ($(32442)$) -- ($(324)$) -- cycle;

\node (34) [circle,fill,inner sep=1.5pt,label=0:{$\Omega^{3,4}$}] at (2.315, 0.1118) {};
\node (13434) [circle,fill,inner sep=1.5pt,label=90:{$\mathrm{L}^{1,2}_{3,4}$}] at ($0.5*(134) + 0.5*(34)$) {};
\draw [color=black] ($(134)$) -- ($(34)$) -- ($(13434)$) -- ($(134)$) -- cycle;

\node (43) [circle,fill,inner sep=1.5pt,label=270:{$\Omega^{4,3}$}] at (1.2542, -1.94889) {};
\node (13443) [circle,fill,inner sep=1.5pt,label=180:{$\mathrm{L}^{1,2}_{4,3}$}] at ($0.5*(134) + 0.5*(43)$) {};
\draw [color=black] ($(134)$) -- ($(43)$) -- ($(13443)$) -- ($(134)$) -- cycle;

\node (234) [circle,fill,inner sep=1.5pt,label=270:{$\Omega^{2,1}$}] at (2.7032, -1.3371) {};
\node (23434) [circle,fill,inner sep=1.5pt,label=0:{$\mathrm{L}^{2,1}_{3,4}$}] at ($0.5*(234) + 0.5*(34)$) {};
\node (23443) [circle,fill,inner sep=1.5pt,label=270:{$\mathrm{L}^{2,1}_{4,3}$}] at ($0.5*(234) + 0.5*(43)$) {};
\draw [color=black] ($(234)$) -- ($(34)$) -- ($(23434)$) -- ($(234)$) -- cycle;
\draw [color=black] ($(234)$) -- ($(43)$) -- ($(23443)$) -- ($(234)$) -- cycle;

\node (V234) at (-3.3757,2.173) {};
\node (V23) at (5.0161,-1.18829) {};
\node (V24) at (3.9885,-3.2657) {};

\node (S2) [star,fill,inner sep=2pt,label=-45:{$\mathrm{S}^{2}$}] at ($0.333*(V234) + 0.333*(23) + 0.333*(24)$) {};
\node (VS2) at ($0.333*(234) + 0.333*(V23) + 0.333*(V24)$) {};
\draw [color=black] ($(24)$) -- ($(S2)$);
\draw [color=black] ($(23)$) -- ($(S2)$);
\draw [color=black] ($0.4*(V234) + 0.6*(S2)$) -- ($(S2)$);
\draw [dashed,color=black] ($0.7*(V234) + 0.3*(S2)$) -- ($(S2)$);
\draw [color=black] ($0.6*(234) + 0.4*(VS2)$) -- ($(234)$);
\draw [dashed,color=black] ($0.3*(234) + 0.7*(VS2)$) -- ($(234)$);

\node (V324) at (3.3757,2.173) {};
\node (V32) at (-5.0161,-1.18829) {};
\node (V34) at (-3.9885,-3.2657) {};

\node (S3) [star,fill,inner sep=2pt, label=225:{$\mathrm{S}^{3}$}] at ($0.333*(V324) + 0.333*(32) + 0.333*(34)$) {};
\node (VS3) at ($0.333*(324) + 0.333*(V32) + 0.333*(V34)$) {};
\draw [color=black] ($(34)$) -- ($(S3)$);
\draw [color=black] ($(32)$) -- ($(S3)$);
\draw [color=black] ($0.4*(V324) + 0.6*(S3)$) -- ($(S3)$);
\draw [dashed,color=black] ($0.7*(V324) + 0.3*(S3)$) -- ($(S3)$);
\draw [color=black] ($0.6*(324) + 0.4*(VS3)$) -- ($(324)$);
\draw [dashed,color=black] ($0.3*(324) + 0.7*(VS3)$) -- ($(324)$);

\node (V423) at (0,-4.1212) {};
\node (V42) at (-1.2542,5.29374) {};
\node (V43) at (1.2542,5.29374) {};

\node (S4) [star,fill,inner sep=2pt,label=90:{$\mathrm{S}^{4}$}] at ($0.333*(V423) + 0.333*(43) + 0.333*(42)$) {};
\node (VS4) at ($0.333*(423) + 0.333*(V43) + 0.333*(V42)$) {};
\draw [color=black] ($(42)$) -- ($(S4)$);
\draw [color=black] ($(43)$) -- ($(S4)$);
\draw [color=black] ($0.4*(V423) + 0.6*(S4)$) -- ($(S4)$);
\draw [dashed,color=black] ($0.7*(V423) + 0.3*(S4)$) -- ($(S4)$);
\draw [color=black] ($0.6*(423) + 0.4*(VS4)$) -- ($(423)$);
\draw [dashed,color=black] ($0.3*(423) + 0.7*(VS4)$) -- ($(423)$);

\node [regular polygon,regular polygon sides=3,fill,color=Red,inner sep=2pt] at ($(14)$) {};
\node [regular polygon,regular polygon sides=3,fill,color=Red,inner sep=2pt] at ($(13)$) {};
\node [regular polygon,regular polygon sides=3,fill,color=Red,inner sep=2pt] at ($(134)$) {};
\node [regular polygon,regular polygon sides=3,fill,color=Red,inner sep=2pt] at ($(23)$) {};
\node [regular polygon,regular polygon sides=3,fill,color=Red,inner sep=2pt] at ($(24)$) {};
\node [regular polygon,regular polygon sides=3,fill,color=Red,inner sep=2pt] at ($(234)$) {};
\node [regular polygon,regular polygon sides=3,fill,color=Red,inner sep=2pt] at ($(32)$) {};
\node [regular polygon,regular polygon sides=3,fill,color=Red,inner sep=2pt] at ($(34)$) {};
\node [regular polygon,regular polygon sides=3,fill,color=Red,inner sep=2pt] at ($(324)$) {};
\node [regular polygon,regular polygon sides=3,fill,color=Red,inner sep=2pt] at ($(42)$) {};
\node [regular polygon,regular polygon sides=3,fill,color=Red,inner sep=2pt] at ($(43)$) {};
\node [regular polygon,regular polygon sides=3,fill,color=Red,inner sep=2pt] at ($(423)$) {};

\node [diamond,fill,color=ForestGreen,inner sep=2.4pt] at ($(1423)$) {};
\node [diamond,fill,color=ForestGreen,inner sep=2.4pt] at ($(1432)$) {};
\node [diamond,fill,color=ForestGreen,inner sep=2.4pt] at ($(32423)$) {};
\node [diamond,fill,color=ForestGreen,inner sep=2.4pt] at ($(23423)$) {};
\node [diamond,fill,color=ForestGreen,inner sep=2.4pt] at ($(1324)$) {};
\node [diamond,fill,color=ForestGreen,inner sep=2.4pt] at ($(24324)$) {};
\node [diamond,fill,color=ForestGreen,inner sep=2.4pt] at ($(32442)$) {};
\node [diamond,fill,color=ForestGreen,inner sep=2.4pt] at ($(1342)$) {};
\node [diamond,fill,color=ForestGreen,inner sep=2.4pt] at ($(13443)$) {};
\node [diamond,fill,color=ForestGreen,inner sep=2.4pt] at ($(23443)$) {};
\node [diamond,fill,color=ForestGreen,inner sep=2.4pt] at ($(23434)$) {};
\node [diamond,fill,color=ForestGreen,inner sep=2.4pt] at ($(13434)$) {};

\node [star,fill,color=blue!80!white,inner sep=2.2pt] at ($(S1)$) {};
\node [star,fill,color=blue!80!white,inner sep=2.2pt] at ($(S2)$) {};
\node [star,fill,color=blue!80!white,inner sep=2.2pt] at ($(S3)$) {};
\node [star,fill,color=blue!80!white,inner sep=2.2pt] at ($(S4)$) {};

\node at ($0.5*(S1)+ 0.5*(13) + (0.1,-0.1)$) {$\alpha_{S^1}$};
\node at ($0.5*(S1)+ 0.5*(14) + (0.2,0)$) {$\alpha_{S^1}$};
\node at ($0.5*(S1)+ 0.5*(134) + (-0.1,-0.1)$) {$\alpha_{S^1}$};

\node at ($0.5*(S2)+ 0.5*(23) + (-0.1,0.1)$) {$\alpha_{S^2}$};
\node at ($0.5*(S2)+ 0.5*(24) + (-0.2,0)$) {$\alpha_{S^2}$};
\node at ($0.5*(S2)+ 0.5*(V234) + (-0.1,-0.1)$) {$\alpha_{S^2}$};
\node at ($0.5*(VS2)+ 0.5*(234) + (-0.1,-0.1)$) {$\alpha_{S^2}$};

\node at ($0.5*(S3)+ 0.5*(32) + (0.1,0.1)$) {$\alpha_{S^3}$};
\node at ($0.5*(S3)+ 0.5*(34) + (0.2,0)$) {$\alpha_{S^3}$};
\node at ($0.5*(S3)+ 0.5*(V324) + (-0.1,0.1)$) {$\alpha_{S^3}$};
\node at ($0.5*(VS3)+ 0.5*(324) + (-0.1,0.1)$) {$\alpha_{S^3}$};

\node at ($0.5*(S4)+ 0.5*(42) + (0.1,0.1)$) {$\alpha_{S^4}$};
\node at ($0.5*(S4)+ 0.5*(43) + (-0.1,0.1)$) {$\alpha_{S^4}$};
\node at ($0.5*(S4)+ 0.5*(V423) + (-0.2,0)$) {$\alpha_{S^4}$};
\node at ($0.5*(VS4)+ 0.5*(423) + (-0.2,0)$) {$\alpha_{S^4}$};

\node at ($0.5*(14)+ 0.5*(1423) + (0.1,0.1)$) {$\alpha_{L^{1,4}_{2,3}}$};
\node at ($0.5*(23)+ 0.5*(1423) + (0.1,0.1)$) {$\alpha_{L^{1,4}_{2,3}}$};
\node at ($0.5*(14)+ 0.5*(1432) + (0.3,-0.1)$) {$\alpha_{L^{1,4}_{3,2}}$};
\node at ($0.5*(32)+ 0.5*(1432) + (0.3,-0.1)$) {$\alpha_{L^{1,4}_{3,2}}$};
\node at ($0.5*(23)+ 0.5*(23423) + (0.25,-0.1)$) {$\alpha_{L^{2,3}_{4,1}}$};
\node at ($0.5*(423)+ 0.5*(23423) + (0.2,-0.15)$) {$\alpha_{L^{2,3}_{4,1}}$};
\node at ($0.5*(423)+ 0.5*(32423) + (0.2,0)$) {$\alpha_{L^{3,2}_{4,1}}$};
\node at ($0.5*(32)+ 0.5*(32423) + (-0.2,-0.15)$) {$\alpha_{L^{3,2}_{4,1}}$};

\node at ($0.5*(24)+ 0.5*(1324) + (-0.05,0.15)$) {$\alpha_{L^{1,3}_{2,4}}$};
\node at ($0.5*(13)+ 0.5*(1324) + (0.1,0.1)$) {$\alpha_{L^{1,3}_{2,4}}$};
\node at ($0.5*(13)+ 0.5*(1342) + (-0.25,0)$) {$\alpha_{L^{1,3}_{4,2}}$};
\node at ($0.5*(42)+ 0.5*(1342) + (0.25,0)$) {$\alpha_{L^{1,3}_{4,2}}$};
\node at ($0.5*(42)+ 0.5*(32442) + (0.1,0.15)$) {$\alpha_{L^{3,1}_{4,2}}$};
\node at ($0.5*(324)+ 0.5*(32442) + (0.1,0.15)$) {$\alpha_{L^{3,1}_{4,2}}$};
\node at ($0.5*(324)+ 0.5*(24324) + (-0.25,-0.1)$) {$\alpha_{L^{2,4}_{3,1}}$};
\node at ($0.5*(24)+ 0.5*(24324) + (0.25,-0.1)$) {$\alpha_{L^{2,4}_{3,1}}$};

\node at ($0.5*(34)+ 0.5*(13434) + (0.05,0.2)$) {$\alpha_{L^{1,2}_{3,4}}$};
\node at ($0.5*(134)+ 0.5*(13434) + (-0.1,0.15)$) {$\alpha_{L^{1,2}_{3,4}}$};
\node at ($0.5*(134)+ 0.5*(13443) + (0.25,0)$) {$\alpha_{L^{1,2}_{4,3}}$};
\node at ($0.5*(43)+ 0.5*(13443) + (-0.25,-0.1)$) {$\alpha_{L^{1,2}_{4,3}}$};
\node at ($0.5*(43)+ 0.5*(23443) + (-0.1,0.15)$) {$\alpha_{L^{2,1}_{4,3}}$};
\node at ($0.5*(234)+ 0.5*(23443) + (-0.1,0.15)$) {$\alpha_{L^{2,1}_{4,3}}$};
\node at ($0.5*(234)+ 0.5*(23434) + (0.25,-0.1)$) {$\alpha_{L^{2,1}_{3,4}}$};
\node at ($0.5*(34)+ 0.5*(23434) + (-0.2,-0.1)$) {$\alpha_{L^{2,1}_{3,4}}$};

\end{tikzpicture}

\caption[The Link of $\Theta^0$ with $\Out^0(W_4)$-Equivariant Angles]{Another picture of the link of $\left[\textrm{id},\Theta^0\right]$ in $\K_4$ with angle variables eqivariant with respect to the action of $\Out^0(W_4)$. For the $\Out(W_4)$ case, the picture is the same but we can ignore the indexing on the angles.}
\label{fig:lknuc0}
\end{figure}
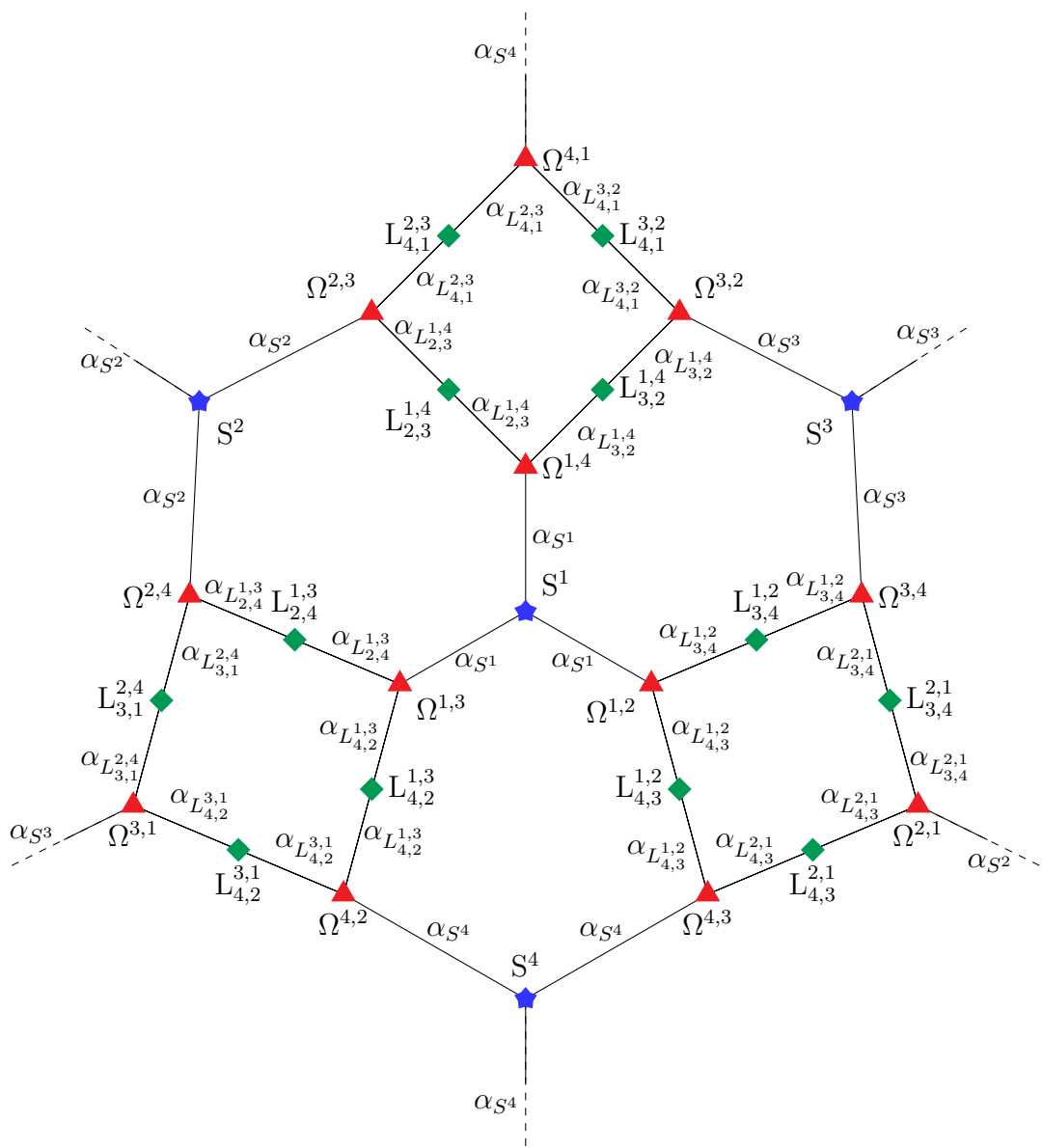

See Figure \ref{fig:lknuc0}. 

\begin{align}
\label{eq:al0}
&\alpha_{L^{i,j}_{k,l}} + \alpha_{L^{k,l}_{i,j}} + \alpha_{L^{k,l}_{j,i}} + \alpha_{L^{j,i}_{k,l}} \\ &\quad + \alpha_{L^{j,i}_{l,k}} + \alpha_{L^{l,k}_{j,i}} + \alpha_{L^{l,k}_{i,j}} + \alpha_{L^{i,j}_{l,k}} \geq 2 \pi \nonumber
\end{align}

\begin{align}
\label{eq:alas0}
&\alpha_{L^{i,j}_{k,l}} + \alpha_{L^{k,l}_{i,j}} + \alpha_{S^{k,l}} + \alpha_{S^{k,j}} \\ &\quad + \alpha_{L^{k,j}_{i,l}} + \alpha_{L^{i,l}_{k,j}} + \alpha_{S^{i,l}} + \alpha_{S^{i,j}} \geq 2 \pi \nonumber
\end{align}

On its own, the system of inequalities \eqref{eq:TL0} - \eqref{eq:alas0} is too complicated to try to solve by hand. However, we can exploit an additional unused symmetry, not of the metric space $K_4$, but of the inequalities themselves, namely that the system is invariant under the action of $\Sigma_4$ that permutes the labels in the subscripts. In fact, we have been implicitly using this symmetry to avoid the explicit quantification over $i,j,k,l \in [4]$ in the different classes of inequalities. 

So now we explicitly note that $\Sigma_4$ acts on the set of 108 angles given in Definition \ref{def:ang0} as follows. For any $\sigma \in \Sigma_4$:
\begin{align*}
\sigma \cdot \alpha_{L^{i,j}_{k,l}} = \alpha_{L^{\sigma(i),\sigma(j)}_{\sigma(k),\sigma(l)}} \quad 
\sigma \cdot \beta_{L^{i,j}_{k,l}} = \beta_{L^{\sigma(i),\sigma(j)}_{\sigma(k),\sigma(l)}} \quad
\sigma \cdot \gamma_{L^{i,j}_{k,l}} = \gamma_{L^{\sigma(i),\sigma(j)}_{\sigma(k),\sigma(l)}} \\
\sigma \cdot \alpha_{S^{i,j}} = \alpha_{S^{\sigma(i),\sigma(j)}} \quad
\sigma \cdot \beta_{S^{i,j}} = \beta_{S^{\sigma(i),\sigma(j)}} \quad 
\sigma \cdot \gamma_{S^{i,j}} = \gamma_{S^{\sigma(i),\sigma(j)}} 
\end{align*}

\begin{definition}
\label{def:ang0d}
Given the angles defined in Definition \ref{def:ang0}, we define the following six \emph{average angles}.
\begin{align*}
\alpha_L := \tfrac{1}{24} \sum_{\sigma \in \Sigma_4} \sigma \cdot \alpha_{L^{1,2}_{3,4}} \quad \alpha_S := \tfrac{1}{24}  \sum_{\sigma \in \Sigma_4} \sigma \cdot \alpha_{S^{1,2}} \\ 
\beta_L := \tfrac{1}{24}  \sum_{\sigma \in \Sigma_4} \sigma \cdot \beta_{L^{1,2}_{3,4}} \quad \beta_S := \tfrac{1}{24} \sum_{\sigma \in \Sigma_4} \sigma \cdot \beta_{S^{1,2}} \\ 
\gamma_L := \tfrac{1}{24} \sum_{\sigma \in \Sigma_4} \sigma \cdot \gamma_{L^{1,2}_{3,4}} \quad \gamma_S := \tfrac{1}{24} \sum_{\sigma \in \Sigma_4} \sigma \cdot \gamma_{S^{1,2}}
\end{align*}
\end{definition}

Note that some angles might appear more than once in these sums as $\Sigma_4$ does not act freely on the set of angles. Also, notice that since Theorem \ref{thm:brid} implies that each angle from Definition \ref{def:ang0} is in ${(0,\pi]}$, then the new average angles in Definition \ref{def:ang0d} are also in ${(0,\pi]}$, since $|\Sigma_4| = 24$. 

We can now prove Theorem \ref{thm:noout0}:

\begin{proof}[Proof of Theorem \ref{thm:noout0}]
For each class of Inequalities \eqref{eq:TL0} - \eqref{eq:alas0}, we take one instance of the inequality for each of the 24 possible assignments of distinct $i,j,k,l$ from $[4] = \{1,2,3,4\}$, and then add the instances together. 

For instance, consider a particular instance of Inequality \eqref{eq:TL0}:
\begin{align*}
\alpha_{L^{1,2}_{3,4}} + \beta_{L^{1,2}_{3,4}} + \gamma_{L^{1,2}_{3,4}} \leq \pi
\end{align*}
Each of the other 23 instances of this inequality is obtained by permuting the labels, i.e., by acting on each variable in the inequality by $\sigma \in \Sigma_4$. When we add the 24 instances of this inequality together, we get
\begin{align*}
\sum_{\sigma \in \Sigma_4} \sigma \cdot \left( \alpha_{L^{1,2}_{3,4}} +  \beta_{L^{1,2}_{3,4}} +  \gamma_{L^{1,2}_{3,4}} \right) &\leq \sum_{\sigma \in \Sigma_4} \pi \\
\implies \sum_{\sigma \in \Sigma_4} \sigma \cdot \alpha_{L^{1,2}_{3,4}} + \sum_{\sigma \in \Sigma_4} \sigma \cdot \beta_{L^{1,2}_{3,4}} + \sum_{\sigma \in \Sigma_4} \sigma \cdot \gamma_{L^{1,2}_{3,4}} &\leq \sum_{\sigma \in \Sigma_4} \pi \\
\implies 24 \alpha_{L} + 24 \beta_{L} + 24 \gamma_{L} &\leq 24 \pi \\
\implies \alpha_{L} + \beta_{L} + \gamma_{L} &\leq \pi,
\end{align*}
i.e., we recover Inequality \eqref{eq:TL}. 

In fact, this is general. For each class of inequalities \eqref{eq:TL0} - \eqref{eq:alas0}, adding together all 24 instances of them indexed by the action of $\Sigma_4$ and then dividing by 24 implies the Inequalities \eqref{eq:TL} - \eqref{eq:alas} in the six average angles variables $\{\alpha_L,\beta_L,\gamma_L,\alpha_S,\beta_S,\gamma_S \}$. 

So assuming the existence of an $\Out^0(W_4)$-equivariant metric on the $M_{\kappa}$-simplicial complex $K_4$ (for $\kappa \leq 0$) allowed us to derive six real numbers $\alpha_L,\beta_L,\gamma_L,\alpha_S,\beta_S,\gamma_S \in {(0,\pi]} $ that simultaneously satisfy Inequalities \eqref{eq:TL} - \eqref{eq:alas}. But the proof of Theorem \ref{thm:noout} shows that no such six numbers exist. This completes the proof. 

\end{proof}

Note that these results do not immediately extend to $\K_n$ for $n \geq 5$ since there is no analogue of Theorem \ref{thm:brid} for higher dimensional $M_{\kappa}$-polyhedral complexes. So the analogous theorem to Theorem \ref{thm:noout} for $n \geq 5$ needs a different approach.

\subsection{The $\Out^0(W_n)$ and $\Out(W_n)$ Case}

To extend the results of this section to a general $n \geq 5$, we first notice that $\K_n$ has $\K_4$ as a full subcomplex which is left invariant by $\Out^0(W_4)$ sitting as a subgroup in $\Out^0(W_n)$.  We then wish to prove the following theorem. 

\begin{theorem}
\label{thm:noout0n}
There does not exist an $\Out^0(W_n)$-equivariant (or $\Out(W_n)$-equivariant) piecewise Euclidean (or piecewise hyperbolic) $\cat(0)$ $\LP\cat(-1)\RP$ metric on $\K_n$ for $n \geq 4$. 
\end{theorem}

Note that Theorem \ref{thm:noout0n} suffices for both the $\Out^0(W_n)$ as well as the $\Out(W_n)$ case, since if there were an $\Out(W_n)$-equivariant piecewise Euclidean (or piecewise hyperbolic) $\cat(0)$ $\LP\cat(-1)\RP$ metric on $\K_n$, then it would be $\Out^0(W_n)$-equivariant as well. 

%By Theorem \ref{thm:pres}, there are higher dimensional analogues to Definition \ref{def:kill}. 

\begin{definition}
\label{def:killn}
Consider the subset $F = \{5,6,\dotsc,n\} \subset [n]$. Denote the partial conjugation generators of $\Out^0(W_n)$ by the usual $x_{i,D}$, and let the partial conjugation generators of $\Out^0(W_4)$ now be denoted as $y_{i,D}$. 

Let $\varphi_{5^{+}} : \Out^0(W_n) \to \Out^0(W_4)$ be defined as 
\[
\varphi_{5^{+}}(x_{i,D}) := 
\begin{cases}
\textrm{id} &\mbox{if } i \geq 5 \textrm{, } D \subset F \textrm{, or } D^c \subset F \\
y_{i,D\setminus F} &\mbox{ otherwise.}
\end{cases}
\]
\end{definition}

\begin{remark}
By checking that each of the relation families (R1), (R2), and (R3) are preserved under the operations of either removing $F$ from $D$ or by sending certain generators to the identity, we can see that each map $\varphi_{5^{+}}$ is a surjective homomorphism onto $\Out^0(W_4)$. 

Furthermore, consider the map: $\psi_{5^{+}} : \Out^0(W_4) \to \Out^0(W_n)$ which is defined as 
$\psi_{5^{+}}(y_{i,D}) := x_{i,D}$. For each $y_{i,D}$, $x_{i,D} = \psi_{5^{+}}(y_{i,D})$ trivially satifies relation families (R1) and (R2), and since $F \subset D^c$ for all images of the map, the disjointness conditions in (R3) remain satisfied as well. (It's critically important here that none of the three disjointness conditions is $\widetilde{D_i^c} \cap \widetilde{D_j^c} = \varnothing$).  Thus, $\psi_{5^{+}}$ is a section of $\varphi_{5^{+}}$, and so $\Out^0(W_n)$ splits as a semidirect product. In particular, it contains $\Out^0(W_4)$ as a subgroup, which by abuse of notation we also denote by $\Out^0(W_4)$.
\end{remark}

Now, we embed $\HTc_4$ into $\HTc_n$. 

\begin{definition}
Let $\Theta \in \HTc_4$ be a hypertree. Then to $\Theta$, associate a hypertree $\widetilde{\Theta} \in \HTc_n$, which is defined to be the hypertree on $[n]$ with the same hyperedges as $\Theta$ as well as the additional hyperedges $\{\{1,f\} \mid f \in F\} = \{\{1,5\},\{1,6\},\dotsc,\{1,n\}\}\}$, i.e., put each remaining vertex in a hyperedge with the vertex $1$. Denote the subset of $\HTc_n$ given by all such $\widetilde{\Theta}$ as $\widetilde{\HTc_4}$.
\end{definition}

\begin{remark}
By adding or removing these hyperedges, we see that there is a bijection between $\HTc_4$ and $\widetilde{\HTc_4}$, this bijection respects folding, and so it is order-preserving from $(\HTc_4, \leq)$ to $(\HTc_n, \leq)$. Thus, it is also a simplicial automorphism from $\HT_4$ into $\HT_n$. 

In order to see how this subcomplex sits in $\K_n$, we need to see which partial conjugations are carried by each hypertree. For each $\widetilde{\Theta} \in \widetilde{\HTc_4}$, if $y_{i,D}$ is carried by $\Theta$, then $x_{i,D}$ is carried by $\widetilde{\Theta}$, since for $i \neq 1$, $1 \in D^c$, and for $i=1$, $F$ is its own union of  connected components of $\widetilde{\Theta} \setminus \{i\}$. This also shows that $\widetilde{\Theta}$ carries $x_{1,F'}$ for any $F' \subseteq F$, which thus commutes with all the other carried partial conjugations by Theorem \ref{thm:comhyp}. If $\Theta$ is at height $h$ and so has $2^h$ carried automorphisms, then $\widetilde{\Theta}$, with its $n-4$ additional hyperedges, is at height $h+n-4$, and so the $2^{h+n-4}$ automorphisms given by $\{x_{i,D} x_{1,F'} \mid x_{i,D} \in \Out^0(W_4) \textrm{ carried by } \Theta, F' \subseteq F\}$ exhaust all the automorphisms carried by $\widetilde{\Theta}$. 
\end{remark}

Next, we embed $\K_4$ into $\K_n$. Consider the subgroup $G = \langle x_{1,\{f\}} \mid f \in F \rangle \subset \Out^0(W_n)$, which is a product of $n-4$ commuting non-conjugate involutions, and so is isomorphic to $\mathbb{Z}_2^{n-4}$. $G$ is thus a finite group acting on $\K_n$ by simplicial automorphisms. 

\begin{theorem}
\label{thm:fix}
The fixed point set of $G$ in $\K_n$ is the set of simplices spanned by $[\alpha, \widetilde{\Theta}]$, where $\alpha \in \Out^0(W_4)$ and $\widetilde{\Theta} \in \widetilde{\HTc_4}$. This set is simplicially isomorphic with $\K_4$.   
\end{theorem}

\begin{proof}
If a simplicial automorphism fixes a simplex pointwise, then it fixes each vertex in that simplex. Conversely, since $\K_n$ is a flag complex, any simplicial automorphism that fixes each of the vertices in a simplex will fix the simplex they span.

So suppose that $[\alpha,\Lambda]$ is a vertex of $\K_n$ that is fixed by every element of $G$, i.e., for each subset $F' \subset F$,
\[
x_{1,F'} \cdot [\alpha, \Lambda] = [x_{1,F'} \alpha, \Lambda].
\]
By the definition of $\K_n$, this happens precisely when $\alpha^{-1} x_{1,F'} \alpha$ is carried by $\Lambda$. But the automorphisms carried by a hypertree are products of pairwise commuting partial conjugations from $\mathcal{P}^0$ (by Theorem \ref{thm:comhyp}), and these commuting products all project injectively into the abelianization of $\Out^0(W_n)$. Thus, $\alpha^{-1} x_{1,F'} \alpha$ must be equal to $x_{1,F'}$, i.e., $\alpha$ commutes with every $x_{1,F'}$.

Additionally, this implies that $\Lambda$ carries $x_{1,\{f\}}$ for each $f \in F$. Thus, $\{f\}$ must be a connected component of $\Lambda \setminus \{1\}$, i.e., $\{1,f\}$ is a hyperedge of $\Lambda$ for each $f \in F$. Therefore, $\Lambda = \widetilde{\Theta}$ for some $\Theta \in \widetilde{\HTc_4}$. 

Now, since $\alpha$ commutes with every $x_{1,F'}$, we claim that $\alpha \in \Out^0(W_4) \times G$. We will induct on the word length of $\alpha$. 

If $\alpha = x_{i,D}$, then we know that $1 \not\in D$ (by our naming convention for $D$). If $i \in F$, then $x_{i,D}$ will not commute with $x_{1,\{i\}}$ by Lemma \ref{lem:compc}, which contradicts our assumption. So $i \not\in F$. If $i \neq 1$, then since $1 \not\in D$, $i \not\in F$, for $x_{i,D}$ to commute with $x_{1,F}$, Lemma \ref{lem:compc} forces $D \cap F = \varnothing$, and so $x_{i,D} \in \Out^0(W_4)$. If $i = 1$, then $x_{i,D} = x_{1,D} x_{1,F'}$, where $D' \cap F = \varnothing$ and $F' \subset F$ (either might be empty). In that case, $x_{i,D}$ is again in $\Out^0(W_4) \times G$.  

Now, inductively assume that $\alpha = \alpha' x_{i,D}$, where $\alpha' = \beta x_{1,F''} \in \Out^0(W_4) \times G$. Then $x_{i,D} = \alpha^{-1} \beta x_{1,F''}$ also commutes with every $x_{1,F'}$. But then by the base case, $x_{i,D} \in \Out^0(W_4) \times G$, and thus so is $\alpha$.

Thus, we now have that every vertex in the fixed point set of $G$ is of the form $[\beta x_{1,F'}, \widetilde{\Theta}]$ for $\beta \in \Out^0(W_4)$ and $F' \subset F$. But since $\beta^{-1} \beta x_{1,F'} = x_{1,F'}$ is carried by each $\widetilde{\Theta}$, we have that in $\K_n$, $[\beta x_{1,F'}, \widetilde{\Theta}] = [\beta, \widetilde{\Theta}]$, and so the fixed point set of $G$ is generated by $[\Out^0(W_4), \widetilde{\HTc_4}]$. Since the carrying partial order of $\widetilde{\HTc_4}$ is isomorphic to $\HTc_4$, we have that the fixed point set of $G$ is a simplicially isomorphic copy of $\K_4$ which admits the same action of $\Out^0(W_4)$. By abuse of notation, we call this subcomplex $\K_4$.  
\end{proof}

Now we can prove the main theorem of the section. 

\begin{proof}[Proof of Theorem \ref{thm:noout0n}]
Suppose that for $\kappa \leq 0$, there existed an $\Out^0(W_n)$-equivariant $\cat(\kappa)$ $M_{\kappa}$-simplicial metric on $K_n$. Since there are only finitely many shapes, the metric is complete (Theorem 7.50 in \cite{BH}). Then the action by $\Out^0(W_n)$ would be by isometries, and so $G$ is a finite group of isometries of the complete $\cat(0)$ space $\K_n$, and so the fixed point set of $G$, namely $\K_4 \subset \K_n$ by Theorem \ref{thm:fix}, is a convex subspace of $\K_n$ (by Corollary 2.8 in \cite{BH}), and so would inherit a $\cat(0)$ $M_{\kappa}$-simplicial metric. Since the metric on $\K_n$ is $\Out^0(W_n)$-equivariant, and since $\Out^0(W_4)$ leaves $\K_4$ invariant, the induced metric on $\K_4$ is $\Out^0(W_4)$-equivariant as well. But this contradicts Theorem \ref{thm:noout0}.
\end{proof}

\section{Future Directions}

\label{sec:futres}

From Section \ref{sec:metmm4}, we know that the natural combinatorial model $\K_n$  of $\Out^0(W_n)$ cannot prove the group to be $\cat(0)$. So now we are left with two options.

\begin{enumerate}
\item \label{fut1} If $\Out^0(W_n)$ is $\cat(0)$, then we will need to investigate a different geometric model space in order to prove it. 
\item \label{fut2} If $\Out^0(W_n)$ is \emph{not} $\cat(0)$, then perhaps that can be detected with known invariants of $\cat(0)$ geometry. 
\end{enumerate}

%While Option \eqref{fut1} is an interesting area for future research, in this section, we focus on Option \eqref{fut2}.

Both options are interesting areas for future research. In particular, all $\cat(0)$ groups and $\cat(0)$ metric spaces are known to satisfy an at most quadratic isoperimetric inequality \cite{BH}. Since isoperimetric inequality is a quasi-isometry invariant, we can study it either directly in the group $\Out^0(W_n)$ or in the model $\K_n$ by endowing it with any $\Out^0(W_n)$-equivariant metric, such as by declaring every edge to have length 1 and then taking the induced path metric. This turns all simplices into equilateral Euclidean simplices. This metric won't be $\cat(0)$ as Theorem \ref{thm:noout0n} promises, but it is still quasi-isometric to $\Out^0(W_n)$ via the action, and so will have the same optimal class of isoperimetric inequalities. Thus, we wish to in the future compute the isoperimetric inequality of either $\K_n$ or else $\Out^0(W_n)$ directly by more combinatorial and geometric methods. In particular, we will need to find a normal form for $\Out^0(W_n)$ and calculate its algorithmic and combinatorial group theoretic properties.

\newpage
\bibliography{Cunningham}{}
\bibliographystyle{amsalpha}

\end{document}